\documentclass[journal,12pt,onecolumn]{IEEEtran}

\usepackage{amsmath,amsthm,amsfonts,amssymb,amscd}
\usepackage{mathrsfs,bm}
\usepackage{mathtools}

\renewcommand{\theenumi}{\alph{enumi}}

\usepackage{lastpage}
\usepackage{enumerate}
\usepackage{fancyhdr}
\usepackage{mathrsfs}
\usepackage{graphicx}
\usepackage{listings}
\usepackage{setspace}
\usepackage{algorithm, algorithmic}
\usepackage{braket}
\newcommand{\prox}{\mathrm{prox}}

\newcommand{\ngb}{ u\in\calN^v_{\text{in}}\cup \{v\} }
\newcommand{\slfrac}[2]{\left.#1\middle/#2\right.}

\newcommand{\epi}{\mathrm{epi}}

\theoremstyle{definition}
\newtheorem{definition}{Definition} 
\newtheorem{assumption}{Assumption} 
\newtheorem{remark}{Remark} 

\DeclareMathAlphabet\mathbfcal{OMS}{cmsy}{b}{n}

\usepackage{array, cite}
\usepackage{float}
\usepackage{multirow}
\usepackage{multicol}
\usepackage[caption=false,font=footnotesize]{subfig}
\usepackage{textcomp}
\def\BibTeX{{\rm B\kern-.05em{\sc i\kern-.025em b}\kern-.08em
    T\kern-.1667em\lower.7ex\hbox{E}\kern-.125emX}}

\newcommand{\bDelta}{\bm{\Delta}}
\newcommand{\bLambda}{\bm{\Lambda}}

\newcommand{\bPsi}{\bm{\Psi}}

\newcommand{\bdelta}{\bm{\delta}}

\newcommand{\bgamma}{\bm{\gamma}}

\newcommand{\blambda}{\bm{\lambda}}

\newcommand{\bpsi}{\bm{\psi}}

\newcommand{\bpi}{\bm{\pi}}

\newcommand{\bnabla}{\boldsymbol{\nabla}}

\newcommand{\tlbG}{\tilde{\bG}}
\newcommand{\tlbX}{\tilde{\bX}}

\newcommand{\tlbg}{\tilde{\bg}}

\newcommand{\bA}{\mathbf{A}}
\newcommand{\bD}{\mathbf{D}}
\newcommand{\bE}{\mathbf{E}}
\newcommand{\bF}{\mathbf{F}}
\newcommand{\bG}{\mathbf{G}}
\newcommand{\bH}{\mathbf{H}}
\newcommand{\bI}{\mathbf{I}}

\newcommand{\bL}{\mathbf{L}}
\newcommand{\bM}{\mathbf{M}}

\newcommand{\bO}{\mathbf{O}}
\newcommand{\bP}{\mathbf{P}}
\newcommand{\bQ}{\mathbf{Q}}
\newcommand{\bR}{\mathbf{R}}
\newcommand{\bS}{\mathbf{S}}
\newcommand{\bT}{\mathbf{T}}
\newcommand{\bW}{\mathbf{W}}
\newcommand{\bX}{\mathbf{X}}

\newcommand{\bZ}{\mathbf{Z}}

\newcommand{\bU}{\mathbf{U}}
\newcommand{\bC}{\mathbf{C}}
\newcommand{\bB}{\mathbf{B}}

\newcommand{\ba}{\mathbf{a}}
\newcommand{\bb}{\mathbf{b}}
\newcommand{\bc}{\mathbf{c}}
\newcommand{\bt}{\mathbf{t}}
\newcommand{\bd}{\mathbf{d}}

\newcommand{\mbf}{\mathbf{f}}
\newcommand{\bg}{\mathbf{g}}
\newcommand{\bh}{\mathbf{h}}

\newcommand{\bn}{\mathbf{n}}

\newcommand{\bu}{\mathbf{u}}

\newcommand{\bw}{\mathbf{w}}
\newcommand{\bx}{\mathbf{x}}
\newcommand{\by}{\mathbf{y}}
\newcommand{\bz}{\mathbf{z}}

\newcommand{\bbR}{\mathbb{R}}

\newcommand{\calE}{\mathcal{E}}

\newcommand{\calG}{\mathcal{G}}

\newcommand{\calN}{\mathcal{N}}

\newcommand{\calV}{\mathcal{V}}

\newcommand{\calO}{\mathcal{O}}

\newcommand{\tlf}{\tilde{f}}

\newcommand{\barbx}{\bar{\bx}}
\newcommand{\barby}{\bar{\by}}

\newcommand{\tlbx}{\tilde{\bx}}

\newcommand{\bzero}{\mathbf{0}}
\newcommand{\bone}{\mathbf{1}}

\newcommand{\suml}{\sum\limits}
\newcommand{\minl}{\min\limits}
\newcommand{\maxl}{\max\limits}

\newcommand{\bigcapl}{\bigcap\limits}

\newcommand{\tr}{\text{Tr}}

\newcommand{\tth}{\text{th}}

\newtheorem{theorem}{Theorem}
\newtheorem{lemma}{Lemma}

\newcommand{\eg}{{{\em e.g.}}}

\newcommand{\nf}[1]{\|#1\|_\mathrm{F}} 
\newcommand{\nt}[1]{\|#1\|_2}

\newcommand{\sbjt}{\mbox{{s.t.}}}

\DeclareMathOperator*{\argmin}{arg\,min}

\newcommand{\brond}{{\mbox{\boldmath $\partial$}}}
\begin{document}
\title{Double Averaging and Gradient Projection:\\Convergence Guarantees for Decentralized Constrained Optimization}
\author{Firooz Shahriari-Mehr, 
and Ashkan Panahi%
\thanks{
This work has been submitted to the IEEE for possible publication. Copyright may be transferred without notice, after which this version may no longer be accessible.
This work was partially supported by the Wallenberg AI, Autonomous Systems and Software Program (WASP) funded by the Knut and Alice Wallenberg Foundation. 
A short version of this work was presented at the 60th IEEE Conference on Decision and Control, Austin, Texas, USA, December 2021~\cite{firoozCDC}.}
\thanks{The authors are with the department of Computer Science and Engineering, Chalmers University of Technology, G\"oteborg, Sweden (e-mails: \{Firooz, Ashkan.panahi\} @chalmers.se).  }
}

\maketitle

\begin{abstract}
We consider a generic decentralized constrained optimization problem over static, directed communication networks, where each agent has exclusive access to only one convex, differentiable, local objective term and one convex constraint set. 
For this setup, we propose a novel decentralized algorithm, called DAGP (Double Averaging and Gradient Projection), based on local gradients, projection onto local constraints, and local averaging.  We achieve global optimality through a novel distributed tracking technique we call distributed null projection. Further, we show that DAGP can be used to solve unconstrained problems with non-differentiable objective terms with a problem reduction scheme.
Assuming only smoothness of the objective terms, we study the convergence of DAGP and establish 
 sub-linear rates of convergence in terms of feasibility, consensus, and optimality, with no extra assumption (e.g. strong convexity). For  the analysis, we forego the difficulties of selecting Lyapunov functions  by proposing a new methodology of convergence analysis in optimization problems, which we refer to as aggregate lower-bounding. To demonstrate the generality of this method, we also provide an alternative convergence proof for the standard gradient descent algorithm with smooth functions. Finally, we present numerical results demonstrating the effectiveness of our proposed method in both constrained and unconstrained problems. In particular, we propose a distributed scheme by DAGP for the optimal transport problem with superior performance and speed.

\end{abstract}                        


\section{Introduction}
\label{sec: Intro}
A wide variety of applications involving multi-agent systems concern distributed optimization problems. In these applications, the system consists of  $M$ interacting agents designed to efficiently minimize a global objective function $f(\bx)$ of a common set of $m$ optimization variables $\bx\in\bbR^m$ in the presence of a global constraint set $S$. In many cases, the global function and constraint are only partially available to each agent. A fairly general framework in such problems is an optimization problem of the following form
\begin{equation}
\label{eq: dec_constrained}
     \minl_{\bx \in \bbR^{m}}  \;\;  \frac{1}{M}\suml_{v=1}^{M} f^v(\bx) \qquad
     \sbjt  \quad \bx\in \bigcapl_{v=1}^{M} S^v,
\end{equation}
where the local objective functions $f^v(\bx)$ and the local constraint sets $S^v$ are exclusively available to their corresponding agent. 
There are multiple motivating applications for this setup.
Parameter estimation and resource allocation in sensor networks~\cite{bazerque2009distributed, kar2012distributed},~\cite[chapter~10]{lesser2003distributed}, 
fitting models to big datasets~\cite{cortes1995support, bishop2006pattern}, 
smart grid control~\cite{alizadeh2012demand},
and optimal transport~\cite{torres2021survey} are but a few examples.

The standard single-machine optimization algorithms such as projected gradient descent can solve many instances of the problem in \eqref{eq: dec_constrained}. However, in large-scale or privacy-sensitive applications,
only distributed optimization methods 
are applicable.  In a distributed setup, two different architectures can be considered: centralized, where a central node (agent) coordinates all other agents (workers), 
and the decentralized architecture, where there is no central coordinator. To avoid  master failure and bottleneck problems, decentralized architectures have attracted more attention, in the past decade~\cite{rabbat2018survey, yang2019survey, Xin2020, Nedic2020}.

In many decentralized optimization schemes, each node~$v$ operates on its own  realization $\bx^v$ of the optimization variables.  The nodes are to find a common minimizer 
by communicating their  information  through a communication network and computing weighted averages among the neighbors, represented by so-called gossip matrices~\cite{boyd2006randomized}. Depending on the nature of communication, the  network is represented by either an undirected or directed graph where the edges denote available communication links between the agents. In this paper, we consider 
gossip matrices with directed graphs.

During the past decade, there has been prolific research on decentralized, unconstrained optimization algorithms~\cite{Xin2020a, Nedic2020}. Many well-known algorithms are developed with provable convergence rates under different assumptions on the objective functions, communication graphs, and the step-size. 
A state-of-the-art example is the Push-Pull method~\cite{Pu2021, Xin2018}, utilizing two gossip matrices and a fixed step-size\footnote{In general algorithms may use a fixed or a diminishing step size. In practice, fixed step sizes are often superior as diminishing step sizes require a careful design of a step size schedule, which substantially slows down the underlying algorithm.} for handling directed communication graphs.  The optimality gap in Push-Pull is proven to decreases with a linear rate,  when the local objective functions are smooth and strongly-convex. In the absence of strong convexity and smoothness, no achievable rate of convergence is known for a directed graph.

Compared to the unconstrained setup, the constrained problem, especially with individual constraints as in \eqref{eq: dec_constrained}, is less studied.
Several papers consider a simplification by a shared (non-distributed) constraint and utilize its orthogonal projection operator to find a feasible solution  \cite{xi2016distributed,8,9}. 
Recently, \cite{cheng2022distributed} attempted to address \eqref{eq: dec_constrained} by adapting the Push-Pull algorithm. They demonstrate that the main underlying strategy (known as gradient tracking) of Push-Pull may not be applied with a fixed step size. However, with a diminishing step-size they show that
the feasibility gap decreases at a sub-linear rate, under the assumption of smooth and strongly convex functions, but do not discuss the optimality gap. Alternative strategies with a fixed step size are yet to be discovered. More generally, algorithms with generic convergence properties in the distributed constrained framework of \eqref{eq: dec_constrained} are still lacking. One main reason is that in the constrained case, the standard method of analysis based on the Lyapunov (potential) functions becomes complicated. A decaying step-size can simplify the analysis, but it often results in a dramatic reduction in the convergence speed~\cite{xi2016distributed}.
%
In this paper, we address the above limitations in the distributed convex optimization literature. Our main contributions are summarized as follows.
\subsection{ Contributions}

\begin{itemize}
    \item  
    We propose a novel algorithm, called Double Averaging and Gradient Projection (DAGP), that solves \eqref{eq: dec_constrained}.
    Our method considers a directed communication graph with individual constraints at each node. DAGP leverages two gossip matrices and a fixed step-size, ensuring fast convergence to the optimal solution. 
    
    \item  
    We introduce a novel general convergence analysis framework that we refer to as \emph{aggregate lower-bounding (ALB)}. It foregoes the need for Lyapunov functions and  decaying step-sizes. We showcase the power of ALB by presenting an alternative analysis of Gradient Descent (GD).

    \item  
    Using ALB, we prove that under smoothness of the objective terms, the feasibility gap of DAGP vanishes with  $\calO(1/K)$, and the optimality gap decays with  $\calO(1/\sqrt{K})$, where $K$ denotes the total iterations. ALB lets us put the restrictive assumptions, such as identical local constraints, a decaying step-size, and strong-convexity, aside. 
    
    \item   
    We present a reduction of decentralized unconstrained optimization problems with a non-smooth objective to an equivalent constrained problem as in \eqref{eq: dec_constrained}, with a linear (hence smooth) objective.
    Using this reduction, DAGP can also solve generic non-smooth decentralized optimization problems. In this way, we provide first convergence guarantees for a decentralized setup without smoothness and strong convexity assumptions. 
    
    \item  
    We explore the performance of DAGP in practical applications, particularly in the context of Optimal Transport problem (OT). We present the first decentralized OT formulation, efficiently computing exact sparse solutions for large-scale instances. 
    
\end{itemize}



\subsection{ Literature Review}
\label{sec: literature_review}
Various decentralized optimization methods have been proposed in the literature for different scenarios and underlying assumptions. For a comprehensive survey, see ~\cite{rabbat2018survey, yang2019survey, Xin2020, Nedic2020}. In this section, we review first-order decentralized optimization algorithms for both constrained and unconstrained optimization problems. We ignore other possible extensions, such as consideration of time delay~\cite{wang2018distributed}, local functions with finite-sum structure~\cite{Mokhtari2016,Hendrikx2020a,Xin2020}, message compression~\cite{Koloskova2019, beznosikov2020biased}, time-varying graphs\cite{nedic2014distributed, nedic2017achieving}, or continuous-time methods~\cite{contTime}, as they fall outside the scope of this paper. 
We also ignore decentralized dual-based methods as they require  computationally costly oracles, such as the gradient of conjugate functions, which may not be feasible in generic applications.

\subsubsection{Decentralized unconstrained optimization methods}
Earlier studies on decentralized optimization have considered undirected communication graphs~\cite{nedic2009distributed, lobel2010distributed}. In this setup, doubly-stochastic gossip matrices compatible with the graph structure can be constructed easily.
These methods are restricted to a decaying step-size for convergence. As a result, the provable convergence rate is $\calO({1}/{\sqrt{K}})$, for the setup with convex and smooth objective functions and $\calO({1}/{K})$, when the objective functions become strongly-convex. 
The next group of algorithms tracks the gradient of the global function over iterations for convergence to the optimal solution using fixed step-sizes.
Gradient tracking can be implemented based on the dynamic average consensus protocol~\cite{zhu2010discrete}.  Pioneered by~\cite{xin2018linear},
the optimization methods that employ this approach are DIGing~\cite{nedic2017achieving}, EXTRA~\cite{shi2015extra}, and NEXT~\cite{Lorenzo2016NEXTIN}.
These methods use fixed step-sizes and achieve linear convergence, that is $\calO(\mu^K)$ for a $\mu \in (0,1)$, in a strongly-convex and smooth setting. 

The construction of double-stochastic gossip matrices compatible with directed graphs 
is not straightforward
\cite{gharesifard2012distributed}.
Therefore, practical optimization algorithms over directed graphs use row-stochastic or column-stochastic gossip matrices. 
This modification makes it harder to achieve a consensus and optimal solution, %
 In response, the push-sum protocol~\cite{kempe2003gossip} is introduced. 
The methods~\cite{tsianos2012push,nedic2014distributed} based on the push-sum protocol still require a decaying step-size for convergence. 
\cite{dextra}, \cite{nedic2017achieving}, \cite{11}, and \cite{Xi2018} combine the push-sum protocol and the gradient tracking techniques and respectively proposed the DEXTRA, Push-DIGing, SONATA,  and ADD-OPT algorithms, which are working with fixed step-sizes. These algorithms achieve a linear rate of convergence in a smooth and strongly-convex setting.
%
All these methods use only column-stochastic gossip matrices.
Recently, the Push-Pull~\cite{Pu2021, Xin2018} algorithm has been proposed that utilizes both row-stochastic and column-stochastic matrices.
It utilizes a fixed step-size and has a linear convergence rate in the strongly-convex and smooth setting. 
 Accelerated version of the Push-Pull algorithm is proposed in \cite{nguyen2023accelerated}.
Our algorithm has a similar requirement of the underlying gossip matrices to the Push-Pull algorithm.

\subsubsection{Decentralized constrained optimization methods}
We focus on methods that consider orthogonal projection operators onto the constraint sets. These methods can 
be divided into two groups based on their constraint structure. 
The first group ~\cite{ram2010distributed, xi2016distributed, 8, 9} considers an identical constraint $S$ available to every agent. These methods differ in the graph connectivity assumption~\cite{xi2016distributed}, the optimization approach~\cite{9}, the global objective function~\cite{7}, or  functions characteristics~\cite{8}.

The second group corresponds to the setup where each agent knows its own local constraint $S^v$. \cite{5} and \cite{nedic2010constrained} proposed projection onto local constraint sets in each iteration. They considered undirected graphs and proved a sub-linear rate of convergence for the optimality gap using  decaying step-sizes, only in two special cases: when the constraints are identical, or when the graph is fully connected~\cite{nedic2010constrained}. No rate for the feasibility gap is established. 
\cite{6} extends \cite{5} and \cite{nedic2010constrained} to the setup with noisy communication links and noisy (sub)gradients. 
DDPS~\cite{xi2016distributed} uses two row-stochastic and column-stochastic matrices, but it requires a decaying step-size
 and assumes identical constraints. 
\cite{cheng2022distributed} modifies the Push-Pull algorithm for this problem. They show that a fixed-step size prevents their approach from reaching a fixed-point; thus, they employ a decaying step-size.
To our knowledge, our algorithm is the first one for constrained optimization with different local constraints on directed graphs using a fixed step-size.

\subsection{Paper organization}

In Section \ref{sec: proposed}, we present our problem setup,  proposed method, and the reduction technique to analyze non-smooth decentralized optimization problems. Section \ref{sec: convergence} introduces our new analytical framework for general optimization algorithms with convex and smooth objectives, based on which we analyze DAGP and GD. Finally,  section \ref{sec: exp_results} presents the efficiency of DAGP in
several experiments.

\subsection{Notation}
Bold lowercase and uppercase letters denote vectors and matrices, respectively, with matrix elements of $\bW$ represented as $w_{vu}$.
$\bone_n$ and  $\bzero_n$ respectively denote the $n-$dimensional vectors of all ones and zeros, while $\bO_{m\times n}$ denotes a $m\times n$ matrix of zeros. The indices $m$ and $n$ may be dropped if there is no risk of confusion.
$\langle.,.\rangle$ is the  Euclidean inner product, and $\langle\bA,\bC\rangle = \tr(\bA\bC^T) $ denotes the matrix inner product, where $\tr(\ldotp)$ is the trace.  $\delta_{k,l}$ is the Kronecker delta. 
Subscripts and superscripts typically indicate iteration numbers and node numbers, respectively, e.g., $\nabla f^v(\bx^v_{k})$ is the gradient at node $v$ and iteration $k$. Matrix representations use vector variables as rows,  \eg, matrix $\bG\in \bbR^{M\times m}$ contains $\bg^v \in \bbR^m$ as rows. In this context, $\bG\in\ker(\bW)$ shows that each column of $\bG$ is an element in the null space of $\bG$.

\subsection{Preliminaries}

\theoremstyle{definition}
\begin{definition}[$L$-Smooth function]
A function $f$ is $L$-smooth if it is differentiable, and its derivative is $L$-Lipschitz. For a convex function $f$, this is equivalent to  
\begin{equation}
f(\by) \leq f(\bx) + \langle \nabla f(\bx), \by - \bx \rangle + \frac{L}{2} \nt{\bx-\by}^2. \quad \forall \bx, \by
\end{equation}
\end{definition}

\theoremstyle{definition}
\begin{definition}[Normal cone and Projection operator]
For a closed convex set $S\subset\bbR^n$, the normal cone of $S$ is given by
\begin{equation*}
\partial I_S(\bx) = 
\begin{cases}
\;\emptyset  & \bx \notin S \\
\left\{\bg \in \bbR^n \vert\;\forall \bz \in S,\;\bg^T(\bz - \bx) \leq 0 \right\} & \bx \in S 
\end{cases}.
\end{equation*}
Moreover, the projection of a vector $\bx \in \bbR^n$ onto $S$ is computed by
$
\mathrm{P}_S(\bx) = \argmin_{\by \in S} \nt{\by -\bx}^2,
$
and  the distance $\nt{\bx-\mathrm{P}_S(\bx)}$ between $\bx$ and $S$ is denoted by $\mathrm{dist}(\bx,S)$.
\end{definition}

\theoremstyle{definition}
\begin{definition}[ Graph Theory]

A directed graph is a pair $\calG = (\calV, \calE)$ of $\calV =\{1,\dots,M\}$ as the node set and $\calE \subseteq \calV \times \calV$  as the (directed) edges. The asymmetric adjacency matrix $\bA = [a_{ij}]$ is computed as $a_{ij}=+1$ if $(i,j) \in \calE$, and $0$ otherwise. In our study, edges represent communication links between nodes:  $j$ can send to  $i$ if $(i,j) \in \calE$. Node  $i$'s incoming and outgoing neighbors are $\calN^i_{\text{in}} = \left\{j \vert (i,j) \in \calE \right\}$ and $\calN^i_{\text{out}} = \left\{ j\vert (j,i)  \in \calE\right\} $, respectively. The in-degree and out-degree are the cardinalities of  $\calN^i_{\text{in}}$ and $\calN^i_{\text{out}}$. For graphs with self-loops, node $i$ is included in $\calN^i_{\text{in/out}}$.
We define two Laplacian matrices: $\bL_{\text{in}} = \bD_{\text{in}} - \bA,$ $ \bL_{\text{out}} = \bD_{\text{out}} - \bA$, where $\bD_{\text{in}}$ and $\bD_{\text{out}}$ are diagonal matrices of the nodes' in-degree and out-degree. These matrices respectively have zero row and zero column sums, and their scaled forms serve as gossip matrices.

\end{definition}

\theoremstyle{definition}
\begin{definition}[Sufficient optimality condition]
$\bx^*$ is  a regular optimal point for 
problem \eqref{eq: dec_constrained} if it satisfies the following optimality condition:
\begin{equation}\label{eq: optimality_cond}
    \bzero \in \sum_{v=1}^M \left( \partial I_{S^v}(\bx^*) + \nabla f^v(\bx^*)\right).
\end{equation}
\end{definition}


\section{Problem Setup and Proposed Method}
\label{sec: proposed}
In this section, we first clarify our problem setup by stating underlying assumptions. 
 %
Next, we present the DAGP algorithm and its constructive blocks.  
Finally, we introduce a technique to address and analyze non-smooth decentralized  optimization problems using our optimization framework.

\subsection{ Problem setup}
 In our setup, $M$ agents interact over a directed network, each having a unique objective function and constraint set.
We proceed by presenting the underlying assumptions:

\begin{assumption}[Objectives and Constraints] \label{ass: obj}
The local objective functions $f^v$ are convex, differentiable, and $L-$smooth for a constant $L>0$. The constraints $S^v$ are closed and convex. 
\end{assumption}
\begin{assumption}[Problem Feasibility] \label{ass: fes}
The optimization problem is feasible and  achieves a finite optimal value $f^*$ at a regular optimal feasible  solution $\bx^*$, that is, $f^*=f(\bx^*)$
\end{assumption}
\begin{assumption}[Gossip Matrices] \label{ass: graph}
 
The communication network, represented by graph $\calG$, is static, strongly connected, and contains self-loops. Two gossip matrices, $\bW$ and $\bQ$, exist with a similar sparsity pattern to the adjacency matrix $\bA$. $\bW$ has a zero row sums, while $\bQ$ has zero column sums.  Similar to~\cite{Pu2021},
we assume that $\ker(\bQ) = \ker(\bW^T)$.

\end{assumption}

\subsection{Proposed Algorithm}

\begin{algorithm}[t!]
\caption{ DAGP Algorithm}
\label{alg: dagp}
	\begin{algorithmic}[1]
		\renewcommand{\algorithmicrequire}{\textbf{Input:}}
		\REQUIRE step size $\mu$, scaling parameters $\alpha$ and $\rho$, and gossip matrices $\bW$ and $\bQ$.
		\STATE Initialize $k=0$. Initialize $\bx^v_0$ randomly, and
             $\bg^v_0$ and $\bh^v_0$ with zero vectors, \hspace{0.01cm} $\forall v\in\calV$.
		\REPEAT 
		\STATE Update $\bz^v, \bx^v, \bg^v$ and $\bh^v$ variables using \eqref{eq: z_update}-\eqref{eq: h_update}.
		\STATE Send the updated tuple $(\bx^v_{k+1}, \; \bh^v_{k+1}- \bg^v_{k+1})$ to all out-neighbors $u\in\calN^v_{\mathrm{out}}$, \hspace{0.01cm} $\forall v\in\calV$.
            \STATE Update iteration index: $k = k+1$.
		\UNTIL{Convergence}
	\end{algorithmic}
\end{algorithm}

 DAGP solves the optimality condition in \eqref{eq: optimality_cond} by splitting it to the following system of equations:
\begin{equation}
    \suml_{v=1}^M\bg^v=\bzero,\quad  \bg^v\in \partial I_{S^v}(\bx^
*) + \nabla f^v(\bx^*).
\end{equation}
At iteration $k$, each node $u$ computes and stores a tuple of variables $(\bx_k^u,\bz_k^u, \bg_k^u,\bh_k^u)$ and broadcasts the pair $(\bx^u_k, \bh^u_k - \bg^u_k)$ to all its out-neighbors. Since our algorithm is synchronized, and the communication is  with no delay and distortion, node $v$ has access to the pair $(\bx^u_k, \bh^u_k - \bg^u_k)$, for all $u\in \calN^v_\mathrm{in}$. 
 Accordingly, the local variables at node $v$ are updated by:
\begin{alignat}{3}
    &\bz^v_{k+1} && = \; && \bx^v_k - \suml_{\ngb} w_{v u}\bx^u_k - \mu \left( \nabla f^v(\bx^v_k) -  \bg^v_k   \right) \label{eq: z_update}\\
    &\bx_{k+1}^v && = \; &&\mathrm{P}_{S^v} \left( \bz^v_{k+1} \right) \label{eq: x_update} \\
    &\bg_{k+1}^v && = \; &&\bg_k^v + \rho \left[ \nabla f^v(\bx^v_k)  - \bg^v_k + \frac{1}{\mu}\left(\bz^v_{k+1} - \bx^v_{k+1} \right)   \right] + \alpha \left( \bh_k^v - \bg_k^v \right) \label{eq: g_update} \\
    &\bh_{k+1}^v && = \; &&\bh_k^v -\suml_{\ngb} q_{v u}(\bh_k^u - \bg_k^u). \label{eq: h_update}
\end{alignat}
Here $w_{vu}, q_{vu}$ are respectively the elements of the gossip matrices $\bW$ and $\bQ$ in Assumption \ref{ass: graph}. The positive constants $\mu,\rho$, and $\alpha$ are design parameters.
 This algorithm consists of the following intuitive operations:
    \subsubsection{Gossip-based Consensus}
     
     This operation is associated with the term $\bx^v_k - \sum_{\ngb} w_{vu}\bx^u_k$ in \eqref{eq: z_update}, which represents the weighted averaging of the received information. The purpose of this operation is to attain consensus, i.e. equal solutions among the nodes, which is a well-known procedure
     \cite{boyd2006randomized}. 

    \subsubsection{Augmented objective minimization}
     
    In \eqref{eq: z_update}, the resulting  vector of the gossip-based consensus operation is combined with the \emph{augmented local descent direction}, i.e. $\nabla f^v (\bx^v) - \bg^v$, scaled by a fixed step-size $\mu$. 
    Using the non-augmented descent direction $\nabla f^v(\bx^v)$ with a fixed step size may not lead to an optimal consensus solution, as it leads to a non-resolvable competition among the nodes over their local objective terms.%
    

    \subsubsection{Local projection}
    In \eqref{eq: x_update}, the resulting vector $\bz^v$ of the procedure in \eqref{eq: z_update} is projected onto the local constraint set $S^v$. Therefore, from the second iteration, the local solutions $\bx^v$ lie in their own local constraint set, and if a consensus is achieved, it should be in the intersection of all constraints, as desired. 

    \subsubsection{Gradient and feasible direction tracking}
     
    Since the information of gradients and normal vectors is distributed among agents, global variables, such as a full gradient or a global feasible direction, cannot be computed locally. Instead, it is common to track them by a so-called tracking protocol~\cite{xin2018linear}.  
    By \eqref{eq: g_update} and \eqref{eq: h_update}, we develop a novel tracking technique using $\bg^v, \bh^v$ variables. 
    Note that the term $\bd^v_k=\nabla f^v(\bx^v_k)  + \frac{1}{\mu}\left(\bz^v_{k+1} - \bx^v_{k+1} \right)$ in \eqref{eq: g_update} is a combination of the local gradients and normal vectors to local constraints, corresponding to individual terms in the optimality condition \eqref{eq: optimality_cond}.

    We may interpret  \eqref{eq: g_update} and \eqref{eq: h_update} as a dynamical system (controller) with state vectors $(\bg^v,\bh^v)$ and the input $\bd^v$, where the goal is to achieve $\bg^v_k=\bd^v_k$ and $\sum_v\bg^v_k=\bzero$. In this case,  when consensus is reached, $\bzero=\sum_v \bg^v= \sum_v (\nabla f^v(\bx^v) + \partial I_{S^v}(\bx^v))$, which provides \eqref{eq: optimality_cond}, hence optimality. Fulfilling this goal in a distributed way is not straightforward. In response, the vector $\bh^v$ is introduced, as we explain next.

    \subsubsection{Distributed null projection}
    To achieve  $\sum_v\bg^v = \bzero$ which we refer to as the \emph{null condition}, we introduce a \emph{distributed projection} of $\bg^v$ variables  onto the space where the null condition is satisfied. Note that by assumption \ref{ass: graph}, $\sum_v \bh^v$ does not change over time in \eqref{eq: h_update}. We further initialize $\bh^v_0$ variables such that $\sum_v \bh^v_0 = \bzero$. The simplest way is to initialize them with zero vectors.  The algorithm is designed to make $\bg^v$ variables converge to $\bh^v$. The term $\alpha(\bh^v-\bg^v)$ in \eqref{eq: g_update} provides a suitable feedback loop to reach this goal.    
    


%

%

\subsection{Non-smooth  Decentralized Optimization}

 Non-smooth decentralized optimization problem has many applications, 
such as decentralized SVM~\cite{lee2013distributed , forero2010consensus} and Basis Pursuit with $\ell_1$ regularization~\cite{chen2001atomic}. 
In this section, we introduce a technique that allows us to efficiently solve and analyze    non-smooth decentralized optimizations through DAGP. 
There is no other existing algorithm  for this scenario.  
%
Any decentralized unconstrained optimization with arbitrary objective functions $f^v$ can be reformulated as a constrained optimization using the epigraph definition \cite{CaE:14}:
\begin{equation}\label{eq:reduce}
        (\bx^*,\bt^*) = \min_{\bx,\bt} \;\; \frac{1}{M} \bone^T\bt \quad\; \sbjt \;\; f^v(\bx) \leq t^v, \; \forall v\in\calV. 
\end{equation}

The program in \eqref{eq:reduce} is an instance of our optimization framework in \eqref{eq: dec_constrained}, and we can apply DAGP to find its solution. 
DAGP requires projection onto the set $\{(\bx^v,t^v)\mid f^v(\bx^v)\leq t^v\}$, which is the epigraph of $f^v$.  As seen, the implementation of DAGP hinges on the assumption that such Epigraph Projection Operators (EPOs) can be reliably evaluated in a reasonable amount of time.%
\cite{chierchia2015epigraphical } and \cite[chapter~6.6.2]{parikh2014proximal} introduce methods to compute EPOs. However, these are not computationally efficient. In Appendix \ref{appendix: EPO}, we demonstrate that evaluating the EPO of a convex function $f$ can be simplified to computing a sequence of proximal operators of $f$, leading to an efficient proximal backtracking scheme for EPOs.

\section{ Convergence Analysis}\label{sec: convergence}

In this section, we offer theoretical guarantees for DAGP's convergence. We first present the fixed-point analysis, showing that any fixed-point is a consensus optimal solution to \eqref{eq: dec_constrained}. 
Then, we introduce a novel \emph{aggregate lower-bounding} methodology for convergence analysis of optimization algorithms. This approach and its difference from the traditional Lyapunov-based analysis are presented in section \ref{subsec: ALB}. Using this methodology, we derive DAGP's convergence rate in section \ref{subsec: dagp}. Further, we showcase the generality of our methodology by providing an alternative proof for GD, presented in section \ref{subsec: gd}.

\subsection{Fixed-point analysis}
 
We begin with the following lemma, followed by the DAGP's fixed-point theorem. 

\begin{lemma}
\label{lemma: consensus}
Under assumption \ref{ass: graph}, i.e. $\ker(\bW^T) = \ker(\bQ)$, we have  $\bQ\bW\bx = \bzero$ if and only if $\bx \in \ker(\bW)$.
\begin{proof}
The forward proof is trivial. For the backward proof, since $\bW\bx\in\ker(\bQ)$, it is in the kernel of $\bW^T$ by assumption~\ref{ass: graph}, we can write $\bW^T\bW\bx = \bzero$. Left multiplying by $\bx^T$, we have $\nt{\bW\bx}^2 = 0$, which shows $\bx\in\ker(\bW)$.
\end{proof}
\end{lemma}

\begin{theorem} 
\label{thm: consensus_opt}
Let Assumption \ref{ass: graph} hold. Any fixed-point of DAGP is an optimal and consensus solution of the decentralized constrained optimization problem in \eqref{eq: dec_constrained}, i.e. $\bx^v = \bx^*$ for all $ v \in \calV$, where $\bx^*$ satisfies the sufficient optimality conditions in~\eqref{eq: optimality_cond}.
\end{theorem}
\begin{proof}
Consider an arbitrary fixed-point of the algorithm, that is $\bx_{k+1}^v = \bx^v_k = \bx^v $, $ \bh^v_{k+1} = \bh^v_k = \bh^v $ and $ \bg^v_{k+1} = \bg^v_k = \bg^v$.
The matrix form of DAGP's fixed-point iteration is:
\begin{gather} 
    \bZ     = \bX - \bW\bX - \mu \left( \bnabla \mbf -  \bG   \right) \label{eq: z2} \\
    \bX     = \mathbfcal{P}_{S} \left( \bZ \right) \label{eq: x2} \\
    \rho\left[\bnabla\mbf-\bG+\frac{1}{\mu}\left(\bZ-\bX\right)\right]+\alpha \left(\bH-\bG \right) = \textbf{O}  \label{eq: g2}\\
     \bQ\left(\bH - \bG\right) = \textbf{O}. \label{eq: h2}
\end{gather}
Left multiplying equations \eqref{eq: g2} and \eqref{eq: z2} by $\bQ$, we have $\bQ\bW\bX = \bzero$. 
Therefore, $\bX \in \ker(\bW)$ by Lemma \ref{lemma: consensus}, which shows a consensus fixed-point, i.e. $\bx^v = \bx, \forall v \in \calV$.
As $\bX \in \ker(\bW)$, we conclude from \eqref{eq: z2} and  \eqref{eq: g2} that $\bG=\bH$. Since \eqref{eq: h_update} is designed to preserve the summation of $\bh^v$s, 
and  each element of $\bH$ is initialized with zero vector, we have 
\begin{equation}
\label{eq: tmp222}
\bone^T\bG = \bone^T\bH = \sum_{v\in \calV} ({\bh^v})^T = \bzero^T. 
\end{equation}
From \eqref{eq: x2}, we have $\bz^v-\bx^v \in \partial I_{S^v}$, for all $v\in\calV$, which can be shorthanded to
$
\bZ-\bX \in \brond\bI_S $.
As $\brond\bI_{S}$ is a cone, therefore invariant to scaling, $\bG - \bnabla\mbf \in \brond\bI_S$. Left multiplying by $\bone^T$, moving all the terms to one side, and considering \eqref{eq: tmp222},
 we observe the consensus fixed-point is an optimal solution.
\end{proof}


%

\subsection{  Aggregate lower-bounding: A general framework for analysis of optimization algorithms}\label{subsec: ALB}

In the classical Lyapunov-based analysis of iterative optimization methods, one seeks a positive-definite function $\Phi$ and a Lyapunov function $L$ that satisfy: 
\begin{equation}\label{eq: frame1}
    L(\bPsi_{k+1}) - L(\bPsi_k) + \Phi(\bPsi_k) \leq 0,
\end{equation}
where $\bPsi$ represents the state vectors (optimization variable and possible auxiliary variables) describing the dynamics of the algorithm. 
Then, by adding inequalities in \eqref{eq: frame1} over iterations up to $K-1$, using the telescopic characteristic of the left hand side, and considering $L(\bPsi_K) \geq 0$, we have
\begin{equation}
    \suml_{k=0}^{K-1}\Phi(\bPsi_k) \leq L(\bPsi_0). \label{eq: l3}
\end{equation}
Usually, $\Phi$ contains several positive terms, the convergence properties of which are sought. This may include the objective function or the distance to the feasible set. If $\Phi$ is convex, we may use Jensen's inequality and conclude that

\begin{equation}\label{eq: frame2}
     \Phi(\bar\by) \leq 
     \frac{1}{K} \suml_{k=0}^{K-1} \Phi(\bPsi_k) \leq \frac{L(\bPsi_0)}{K},
\end{equation}                  
where $\bar\by = \frac 1K\sum_{k=0}^{K-1}\bPsi_k$, yielding the standard $\calO\left(\slfrac{1}{K}\right)$ rate. 

For common unconstrained optimization algorithms, finding a Lyapunov function satisfying \eqref{eq: frame1} is straightforward (often a well-designed quadratic function). For distributed and constrained optimization  algorithms this is not straightforward anymore, 
%
In response, in our methodology,  which we refer to as aggregate lower-bounding (ALB), we drop the requirement for Lyapunov functions. We replace the term related to the changes in the Lyapunov function in \eqref{eq: frame1} by a 
negative function $A_k \coloneqq A(\bPsi_{k+1}, \bPsi_k)$. Then, we can rewrite \eqref{eq: frame1} as 
\begin{equation}\label{eq:aggregate}
    A_k + \Phi(\bPsi_k) \leq  0.
\end{equation}
Following the same procedure as the classical analysis, we sum over all iterations. The key step in our analysis is to show that there exists a lower-bound $-C_1$ for the \emph{aggregate} term $\bar{A}_K \coloneqq \sum_{k=0}^{K-1} A_k$. Then, we obtain
\begin{equation}\label{eq: aggregate_bound}
     \suml_{k=0}^{K-1} \Phi (\bPsi_k)  \leq C_1.
\end{equation}
 Note that $\Phi$ has positive terms, $\bar A_K$ is negative and $C_1 > 0$.
 To bound $\bar{A}_K$, we calculate the \emph{minimum}  value of the \emph{aggregate} term $\bar A_K$, in an abstract optimization problem, over all possible trajectories of parameters $\bPsi_k$ generated by the dynamics of the algorithm. We refer to this approach as ALB.
Then, the convergence rate is established similarly to \eqref{eq: frame2}.
%

As a remark,  the classical approach with the Lyapunov function is a special case of ALB, where $A_k=L(\bPsi_{k+1}) - L(\bPsi_k) $. This leads to $ \bar{A}_K=L(\bPsi_K)-L(\bPsi_0)$,  hence  $C_1\leq L(\bPsi_0)$, which indicates that the ALB provides less restrictive and more generic results than the Lyapunov-based analysis.
%
%
In the following, we analyze DAGP and GD using ALB.

\subsection{Convergence rate of DAGP}\label{subsec: dagp}
We start by stating the main theoretical result of this paper, that is a rate of convergence for both optimality and feasibility gaps of the DAGP algorithm. To begin, 
we state a definition:
\begin{definition}    
\label{ass: big_assumption}
 
Take arbitrary matrices $\bS,\bR,\bP$ and for any complex value $z$ and a positive real value $\beta$ define\footnote{This requires the sizes of $\bS,\bR,\bP$ to be compatible.} 
\begin{equation}\label{eq: F_define}
    \bF_\beta(z) \coloneqq \left[ 
    \begin{array}{cc}
        z\bS & \bI-z\bR^T \\
        z\bI-\bR & -\frac{1}{\beta}\bP\bP^T
    \end{array}
    \right].
\end{equation}
We say that the tuple $(\bS,\bR,\bP)$ is proper if the following statements hold true:
\begin{enumerate}
    \item Denote $ Z_\beta\coloneqq\left\{z\neq 0|\det(\bF_\beta(z)) = 0\right\}$. All Elements $z_i(\beta)\in Z_\beta$ are real and simple zeros of $\det(\bF_\beta(z))$.
    \item $z\bF_\beta^{-1}(z)$ has a finite limit as $z\to 0$.
    \item Denote the basis of the null space of $\bF_\beta(z_i)$ (which is single dimensional) by $\bn_{i,\beta} \coloneqq [ \bn_{i,\beta,\bpsi}^T \;\; \bn_{i,\beta ,\blambda}^T ]^T $.  The vectors $\{\bn_{i,\beta}\}$ are linearly independent.
    \item  The vectors  $\{\bn'_{i,\beta}\coloneqq [ z_i\bn_{i,\beta,\bpsi}^T \;\; \bn_{i,\beta ,\blambda}^T ]^T\}$ are linearly independent.
    \item As $\beta\to0$, the limits of $z_i(\beta)$ exist and are distinct. Furthermore the limits of  $\{\bn_{i,\beta}\}$ and $\{\bn'_{i,\beta}\}$ exist and each form an independent set of vectors.
\end{enumerate} 
\end{definition}
Next, we take positive constants $(\eta,  \rho, \alpha, \mu, \beta)$  and define matrices $
\bR,\bS=[\bS_1 \; \bS_2], \bP$ as follows:

\begin{equation} \label{eq: R_define}
\bR = \left[
\begin{array}{cccc}
        \bO & \bO & \bO & \bO   \\
        \bI & \bO & \bO & \bO  \\
        -\frac{\rho}{\mu}\bI & \frac{\rho}{\mu}(\bI-\bW) & \bI & \alpha \bI \\
        \frac{\rho}{\mu}\bI & -\frac{\rho}{\mu}(\bI-\bW) & \bO & (1-\alpha)\bI-\bQ 
\end{array} \right]
\end{equation} 
\begin{equation} \label{eq: S_define}
\begin{split}
\bS_1 = &\left[
\begin{array}{c}
        \left(1-\frac{L\mu}2\right)\bI-M\eta\left(\bI-\frac 1M\bone\bone^T\right)\\
        -\frac 12(\bI-\bW^T)+\frac {L\mu}2\bI\\
         -\frac{\mu}2\bI  \\
         \bO
\end{array}
\right]
\\
\bS_2 = &\left[
\begin{array}{ccc}
      -\frac 12(\bI-\bW)+\frac {L\mu}2\bI  & -\frac {\mu}2\bI  & \bO\\
      -\frac{L\mu}2\bI & \bO &\bO \\
      \bO & \bO & \bO \\
      \bO& \bO & \bO
\end{array}
\right]
\end{split}
\end{equation}
\begin{equation}\label{eq: P_define}
    \bP = \left[
\begin{array}{cccc}
     \bI  &
     \bO  &
     \bO  &
     \bO 
\end{array}
\right]^T
\end{equation}
Then, our main result is given bellow.
\begin{theorem}  \label{thm: rate}
Let assumptions~\ref{ass: obj},\ref{ass: fes},\ref{ass: graph} hold and the matrices $\bS,\bR,\bP$ above are proper as stated in definition \ref{ass: big_assumption}.
Define $\barbx^v_K=\frac 1K\sum_{k=0}^{K-1}\bx^v_k$ and $\barbx_K=\frac 1M\sum_v\barbx^v_K$. Take $C>-\lambda_{\min}(\bS)$, and let $C_0$
be a constant 
dependent on the initial point.
Then, for a sufficiently large $K$, 
we have: 
\begin{itemize}
    \item \textbf{Consensus:} 
    Time-averaged local solutions converge to the consensus solution $\barbx_K$ as 
    \[ 
    \|\barbx_K-\barbx^v_K\|_2^2\leq\frac{C_0C}{\eta M K}.  \qquad \forall v \in \calV
    \]
    \item \textbf{Feasibility gap:}
    The consensus solution $\barbx_K$ approaches each constraint set with the following rate. 
    \[ 
    \mathrm{dist}^2(\barbx_K, S^v)=\calO(\frac 1K) \qquad \forall v \in \calV 
    \]
    Hence, the squared distance between $\barbx_K$ and the feasible set decays by $\calO(\frac 1K)$.
    \item \textbf{Optimality gap:}
    The objective value converges to the optimal value with the following relation
    \begin{equation*}
    \left|\suml_v f^v(\barbx^v_K)-\suml_v f^v(\bx^*)\right|\leq 
    \frac{C_0C}{\mu K}+\sqrt{\frac{C_0CC_2}{\eta M K}},
    \end{equation*} 
     where $C_2= \sqrt{\sum_v\|\bn^v+\nabla f^v(\bx^*)\|^2}$, with $\bn^v \in \partial I_{S^v}(\bx^*)$, only depending on the optimal solution.
    %
\end{itemize}
\end{theorem}


\begin{remark}[ Convergence rate dependencies]
Based on the constants in Theorem~\ref{thm: rate}, we can determine which parameters affect the rates and how. $C_2$ shows the decentralized variance caused by non-homogeneous functions and constraints over the network,  i.e. it shows how far the local solutions of the nodes (if each node minimizes its own local objective function subject to its own local constraint)  are from the global solution. The higher the difference between local objective functions and constraints, the higher the variance, and the lower the rate. If all the nodes have the same minimizer satisfying the constraints, then $C_2 = 0$, and  both feasibility and optimality gaps decay with $\calO(\slfrac{1}{K})$. The constant $C_0$ depends on the initial point. 
If we start from a close solution to the optimal point, $C_0$ will be small, and convergence will be faster. Constant $C$ depend on $-\lambda_{\min}(\bS)$. Note that $\bS$ depends on the gossip matrices and hence $-\lambda_{\min}(\bS)$ is a property of the communication network\footnote{It is well-known that the smallest non-zero eigenvalue of the Laplacian matrix of a graph reflects its overall connectivity \cite{chung2005laplacians}. One may expect that $\lambda_{\min}(\bS)$ plays a similar role.}. Hence, $C$ establishes a connection between the topology of the communication network  and the convergence rates. Finally, the rates explicitly depend on  $\eta, \mu$. Note that $\eta$ is not a parameter of the algorithm. It should be suitably selected in the theorem. As $\bS$ depends on $\eta$, there is a trade off between smaller $C$ and larger $\eta$.  As expected, larger step-sizes result in a faster decrease in the optimality gap. Additionally, the dependency of $\bS$ on $\eta$ implies  an implicit relation between $C$ and $\mu$ as well.
\end{remark}

\begin{remark}[ Feasible region for hyper parameters]\label{remark: feasible}
%
 
The requirement for the matrices to be proper (definition \ref{ass: big_assumption}) imposes restrictions on hyper parameters. In practice, it is not difficult to satisfy these restrictions and there exists a wide range of suitable parameters $(\eta, \rho, \alpha, \mu, \beta)$.
To show this, we conduct an illustrative numerical experiment. 
We assume a fully-connected setup with $M=4$  nodes, where the local objective functions are $L-$smooth with $L=0.1$. We compute $\bR,\bS$, and $\bP$ matrices, generate $\bF_\beta(z)$ given in \eqref{eq: F_define}, and numerically evaluate $Z_\beta$.
%
To find all feasible combinations of design parameters $(\rho,\alpha,\mu)$, we follow two steps. First, we guess a tuple $(\rho, \alpha, \eta)$, and identify step-sizes $\mu$ that satisfy  the conditions of definition \ref{ass: big_assumption}. In Fig.~\ref{fig: beta-mu}, we present the results for $\rho =10^{-6}, \alpha = 0.5$ and  $\eta=0.1$. We observe that there exists $\mu_0 < 4$ such that for every $\bar\mu<\mu_0$, 
the assumption holds.
In the second step, we choose $\bar\mu = 10^{-2}$ from the previous step, and find all possible combinations of $\alpha,\rho$ that satisfy the assumption. Fig.~\ref{fig: rho-alpha} demonstrates these combinations. For various step sizes $\bar\mu$, we repeat the second step many times to find other feasible regions. 
As seen, many combinations of design parameters guarantee that the matrices are proper.  
%
The presented result is a proof of concept, and a more elaborate study of the hyper-parameters is postponed to a different study. 
%
\begin{figure}
    \centering
	\subfloat[Find $\mu_0$.]{\includegraphics[width = 0.49\linewidth]{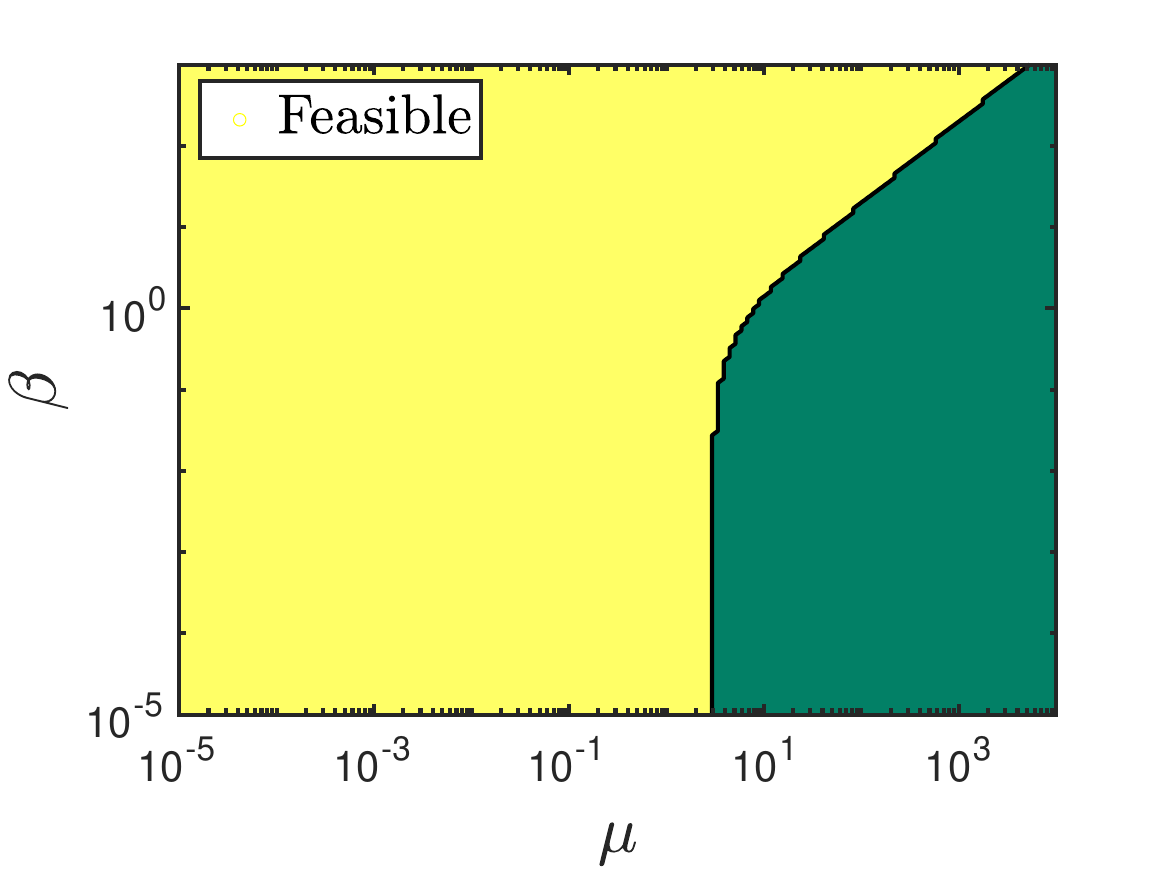} \label{fig: beta-mu}} 
	\subfloat[Feasible $(\rho,\alpha)$.]{\includegraphics[width = 0.49\linewidth]{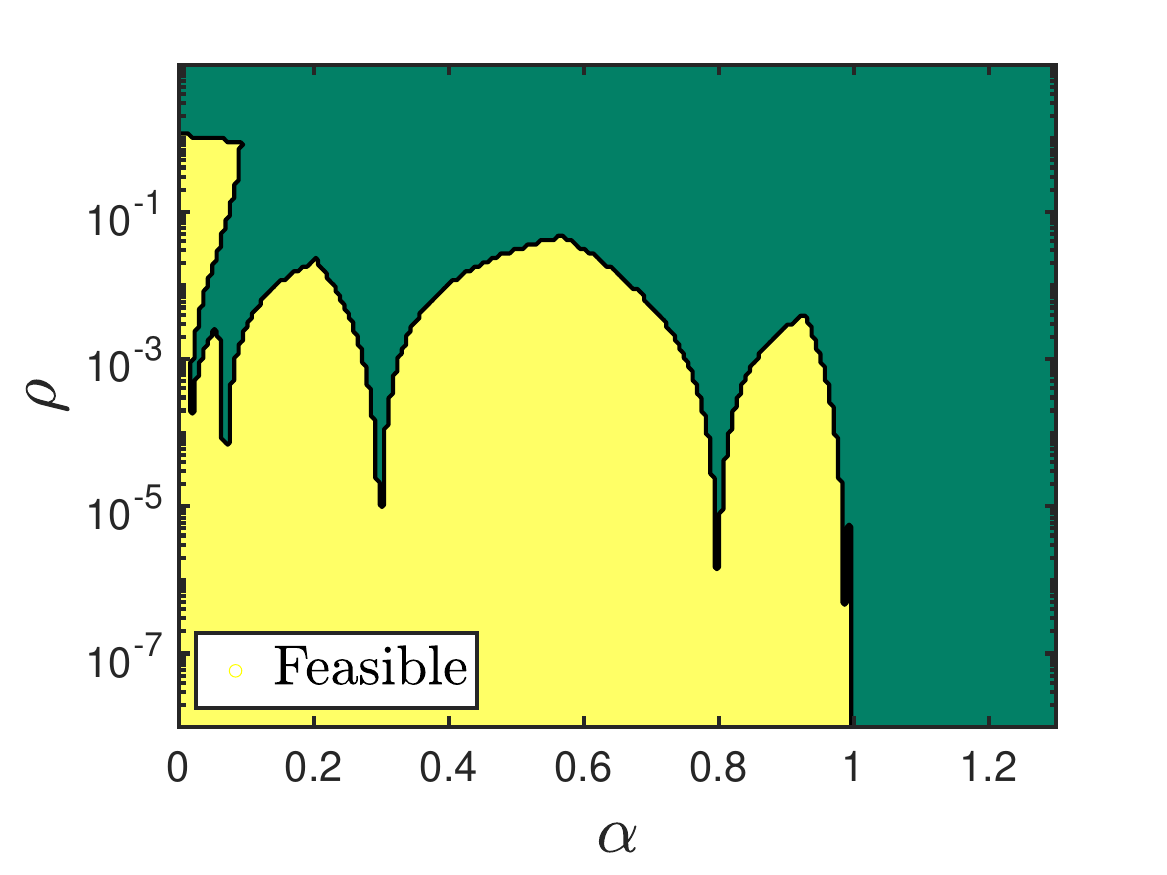} \label{fig: rho-alpha} }
    \caption{ Feasible region for the design parameters $(\rho,\alpha,\mu)$, 
    given the setup presented in the Remark \ref{remark: feasible}.}
    \label{fig: params}
\end{figure}
\end{remark}

\subsection{Proof of Theorem \ref{thm: rate} using ALB}\label{sec: proof}

The proof of Theorem \ref{thm: rate} consists of two major steps. In the first one a function $\Phi(\bPsi_k)$ and an aggregate term $\bar{A}_K$ is derived, as discussed in section \ref{subsec: ALB}.
In the second one, the aggregate term is lower bounded, leading to the desired result in \eqref{eq: aggregate_bound}.
These general steps can also be applied in analyzing other optimization algorithms with small modifications. For illustration, we present that case of GD  in section~\ref{subsec: gd}.
\subsubsection{Deriving the Aggregate Term}
We derive the relation in \eqref{eq:aggregate} and subsequently obtain $\bar{A}_K$ in multiple steps:

{\noindent \bf Combining relations from convexity, smoothness, and projection:} We start by defining
\begin{equation*}
    F^v(\bx) \coloneqq f^v(\bx)-f^v(\bx^*)-\langle\nabla f^v(\bx^*),\bx-\bx^*\rangle
\end{equation*}
and
$F^v_k\coloneqq F^v(\bx^v_k).$ Note that from the convexity of $f^v$, the values of $F^v(\bx)$, particularly $F^v_k$ are non-negative. From convexity, we also conclude that
\begin{alignat}{2}\label{eq:4}
    & F^v_k+\left\langle\nabla f^v(\bx^*)-\nabla f^v(\bx^v_k), \bx^v_k-\bx^*\right\rangle &&= \nonumber \\
    & f^v(\bx^v_k)-f^v(\bx^*)+\left\langle\nabla f^v(\bx^v_k), \bx^*-\bx^v_k\right\rangle&&\leq 0. 
\end{alignat}
From the $L-$smoothness property of $f^v$, we also obtain
\begin{alignat}{2}
        & F_{k+1}^v-F_k^v-\langle\nabla f^v(\bx_k^v)-\nabla f^v(\bx^*),\bx_{k+1}^v-\bx^v_k\rangle && = \nonumber \\
        & f^v(\bx^v_{k+1})-f^v(\bx^v_k)-\langle\nabla f^v(\bx_k^v),\bx_{k+1}^v-\bx^v_k\rangle&& \leq  \nonumber \\
        & \frac L2\left\|\bx_{k+1}^v-\bx_{k}^v\right\|^2. \label{eq:5}
\end{alignat}
Adding \eqref{eq:4} to \eqref{eq:5} yields
\begin{equation}\label{eq:6}
    F^v_{k+1}+\left\langle\nabla f^v(\bx^*)-\nabla f^v(\bx^v_k), \bx^v_{k+1}-\bx^*\right\rangle
    -\frac L2\left\|\bx_{k+1}^v-\bx_{k}^v\right\|^2\leq 0.
\end{equation}
Now, we define
$
T^v(\bx)\coloneqq -\langle\bn^v,\bx-\bx^*\rangle
$
and $T_k^v\coloneqq T^v(\bx^v_{k}),$ where $\bn^v \in \partial I_{S^v}(\bx^*)$. The fact that $\bx_{k+1}^v\in S^v$ yields $T^v_{k+1}\geq 0$. Note that from $\bx^v_{k+1}=\mathrm{P}_{S^v}(\bz^v_{k+1})$ and the fact that $\bx^*\in S^v$, we have
$
\langle\bx^*-\bx^v_{k+1},\bz_{k+1}^v-\bx^v_{k+1}\rangle\leq 0,
$
which can also be written as
\begin{equation}\label{eq:3}
    \mu T^v_{k+1}+\langle\bx^*-\bx^v_{k+1},\bz_{k+1}^v-\bx^v_{k+1}-\mu\bn^v\rangle\leq 0.
\end{equation}
Multiplying \eqref{eq:6} by $\mu$, adding to \eqref{eq:3}, we have
\begin{equation}\label{eq:sumall}
    \mu(F^v_{k+1}+T^v_{k+1}) -\frac{L\mu}{2}\left\|\bx^v_{k+1}-\bx^v_k\right\|+  \Big\langle\bx^*-\bx^v_{k+1},  \bz^v_{k+1}-\bx^v_{k+1} + \mu\big(\nabla f^v(\bx^v_k)-\nabla f^v(\bx^*) - \bn^v\big)\Big\rangle \leq 0.
\end{equation}
\noindent {\bf Plugging algorithm's dynamics into existing relations:}
To keep the notation compact, we set $\sum_{\ngb}(.) = \sum_u(.)$. Now, we plug the definition of $\bz^v_{k+1}$ into \eqref{eq:sumall} as
\begin{multline}
    \mu \left(F_{k+1}^v+T_{k+1}^v\right)  -\frac {L\mu}2\left\|\bx_{k+1}^v-\bx_{k}^v\right\|^2 
     + \Big\langle\bx^*-\bx^v_{k+1}, \bx_k^v-\\\suml_{u}w_{vu}\bx^u_k-\bx^v_{k+1}  +\mu(\bg^v_k-\nabla f^v(\bx^*)-\bn^v)\Big\rangle\leq 0. \label{eq:7}
\end{multline}
We also replace the expression of $\bz^v_{k+1}$ in \eqref{eq: g_update}, leading to:
\begin{equation}\label{eq:8}
    \bg_{k+1}^v=\bg_k^v+\frac\rho\mu\left(\bx_k^v-\suml_{u}w_{vu}\bx^u_k-\bx^v_{k+1}\right)+\alpha\bdelta_k^v,
\end{equation}
where $\bdelta_k^v=\bh_k^v-\bg_k^v$ and follows the following dynamics
\begin{equation}\label{eq:9}
      \bdelta_{k+1}^v= (1-\alpha)\bdelta_{k}^v-\suml_uq_{vu}\bdelta^u_k -\frac\rho\mu\left(\bx_k^v-\suml_{u}w_{vu}\bx^u_k-\bx^v_{k+1}\right).  
\end{equation}
\noindent {\bf Creating Lyapunov functions:}
For simplicity, we define $\tlbx_{k}^v\coloneqq \bx^v_k-\bx^*$ and $\tlbg^v_k\coloneqq \bg^v_k-\nabla f^v(\bx^*)-\bn^v$. We can rewrite \eqref{eq:8} and \eqref{eq:9} as
\begin{equation}\label{eq:8p}
    \tlbg_{k+1}^v=\tlbg_k^v+\frac\rho\mu\left(\tlbx_k^v-\suml_{u}w_{vu}\tlbx^u_k-\tlbx^v_{k+1}\right)+\alpha\bdelta_k^v,
\end{equation}
\begin{equation}
    \bdelta_{k+1}^v= (1-\alpha)\bdelta_{k}^v-\suml_uq_{vu}\bdelta^u_k  -\frac\rho\mu\left(\tlbx_k^v-\suml_{u}w_{vu}\tlbx^u_k-\tlbx^v_{k+1}\right), \label{eq:9p}
\end{equation}
Moreover, by plugging the new definitions in \eqref{eq:7}, summing over $v\in\calV$ and $k=0,\ldots,K-1$, and adding and removing $\frac{\eta}2\suml_{u,v}\|\tlbx_{k+1}^u-\tlbx_{k+1}^v\|^2 $, we have
\begin{align}
    &\suml_{k=0}^{K-1}\left(\mu \suml_v\left(F_{k+1}^v+T_{k+1}^v\right)+\frac{\eta}2\suml_{u,v}\|\tlbx_{k+1}^u-\tlbx_{k+1}^v\|^2\right) \nonumber\\
    & - \suml_{k=0}^{K-1}\suml_v\left\langle\tlbx^v_{k+1},\tlbx_k^v-\suml_{u}w_{vu}\tlbx^u_k-\tlbx^v_{k+1}+\mu\tlbg^v_k\right\rangle \nonumber\\
    &  -\frac {L\mu}2 \suml_{k=0}^{K-1}\suml_v \left\|\tlbx_{k+1}^v-\tlbx_{k}^v\right\|^2 -\frac{\eta}2\suml_{k=0}^{K-1}\suml_{u,v}\|\tlbx_{k+1}^u-\tlbx_{k+1}^v\|^2 \leq 0. \label{eq:7p}
\end{align}

The first summation in \eqref{eq:7p} contains several positive-definite functions of interest for convergence. These terms correspond to $\sum_{k=0}^{K-1}\Phi(\bPsi_k)$ in \eqref{eq: l3}.
We will consider the last three summations in \eqref{eq:7p} as  $\bar{A}_K$, i.e. the aggregated term. We show an asymptotic lower bound for $\bar{A}_K$, i.e. we show that there exists a constant $C_1$
such that for a sufficiently large $K$, $\bar{A}_K\geq -C_1$. 
Then, we conclude from \eqref{eq:7p} that 
\begin{equation*}
    \suml_{k=0}^{K-1}\Bigg(\mu \suml_v\left(F_{k+1}^v+T_{k+1}^v\right)
     +\frac{\eta}2\suml_{u,v}\|\tlbx_{k+1}^u-\tlbx_{k+1}^v\|^2\Bigg)\leq C_1.
\end{equation*}

\subsubsection{Bounding the Aggregate Term $\bar{A}_K$}
The main tool for bounding $\bar{A}_K$ is the following lemma, the proof of which is presented in Appendix \ref{appendix: proofLem1}.
\begin{lemma}\label{lem1}
 Let $\bR,\bS$ and $\bP$ be proper according to definition~\ref{ass: big_assumption}. Assume matrices $\{\bPsi_k\}$ are generated by
\begin{equation}
    \label{eq: linear_dynamics}
    \bPsi_{k+1} = \bR\bPsi_k + \bP\tilde\bX_{k+2} ,  \qquad k=0,\ldots,K-2.
\end{equation}
Then, for $C>-\lambda_{\mathrm{min}}(\bS)$  and a sufficiently large $K$, $\bar{A}_K \coloneqq \sum_{k=0}^{K-1}\langle \bPsi_k,\bS\bPsi_k  \rangle$ is bounded below by $ -C\nf{\bPsi_0}^2$. 
\end{lemma}

To find the bound $C_1$, we simply need to write the dynamics in terms of Lemma \ref{lem1}. We start by simplifying the notation in \eqref{eq:8p}, \eqref{eq:9p}, and \eqref{eq:7p}. Let us introduce 
\begin{equation}
    \bPsi_k \coloneqq \left[
\begin{array}{cccc} 
\tlbX_{k+1}^T & \tlbX_{k}^T&\tlbG_{k}^T & \bDelta_{k}^T
\end{array}
\right]^T,
\end{equation}
where $\tlbX_{k}, \tlbG_{k}, \bDelta_{k}$ are matrices with $\tlbx^v_{k}, \tlbg^v_{k}, \bdelta^v_{k}$ as their $v^\tth$ row, respectively. We may write \eqref{eq:8p} and \eqref{eq:9p} as
 the linear dynamical system presented in \eqref{eq: linear_dynamics},
where $\bR$ and $\bP$ are defined in \eqref{eq: R_define} and \eqref{eq: P_define}, respectively.
We may also define
\begin{equation}\label{eq: AK_def}
    \bar{A}_K=\suml_{k=0}^{K-1}\left\langle\bPsi_k,\bS\bPsi_k\right\rangle,
\end{equation}
where $\bS $ is defined in \eqref{eq: S_define}.
 Since there exists a set of constants $(\eta,  \rho, \alpha, \mu, \beta)$ such that the defined $\bR, \bS$, and $\bP$ matrices satisfy Assumption \ref{ass: big_assumption}, $\bar{A}_K$ is lower bounded by $-C_1 \coloneqq -C\nf{\bPsi_0}^2$ by Lemma~\ref{lem1}.

{\noindent\textbf{Providing the rates of convergence:}}
Lastly, by defining $\barbx^v_K=\frac 1K\sum_{k=0}^{K-1}\bx^v_k$ and noting that each term in the summation over $k$ is a fixed convex function of $\{\bx_{k+1}^v\}_v$, we may recall Jensen's inequality to conclude
 \begin{equation}
 \mu \suml_v\left(F^v(\barbx^v_K)+T^v(\barbx^v_K)\right)
 +\frac{\eta}2\suml_{u,v}\|\barbx_{K}^u-\barbx_{K}^v\|^2
    \leq \frac{C_1}{K}.
\end{equation}
By defining $\barbx_K=\frac 1M\sum_v\barbx^v_K$ and considering $C_1 = C\nf{\bPsi_0}^2$, we conclude that for all $v$, 
\[
\|\barbx_K-\barbx^v_K\|^2 \leq \frac{C_0C}{\eta MK}, 
\]
where constant $C_0 = \nf{\bPsi_0}^2$ depends on the distance between the start point of the algorithm and the optimal solution.  Since $\barbx^v_K\in S^v$, we also conclude that 
\[
\mathrm{dist}^2(\barbx_K, S^v) \leq  \frac{C_0C}{\eta MK}. 
\]
Finally,
\begin{alignat}{2}
    &\bigg|\suml_v && f^v(\barbx^v_K)-\suml_vf^v(\bx^*)\bigg| \nonumber \\ & && \leq \frac{C_0C}{\mu K}+\suml_v\left|\langle\bn^v+\nabla f^v(\bx^*),\barbx_K^v-\barbx_K\rangle\right| \nonumber\\
    & && \leq\frac{C_0C}{\mu K}+\sqrt{\suml_v\|\bn^v+\nabla f^v(\bx^*)\|^2}\sqrt{\suml_v\|\barbx_K^v-\barbx_K\|^2} \nonumber \\
    & && \leq \frac{C_0C}{\mu K}+\sqrt{\frac{C_0CC_2}{\eta M K}}. \nonumber \qed
\end{alignat}

\subsection{Convergence rate of Gradient Descent}\label{subsec: gd}
 
In this section, we present an alternative convergence proof for the well-known GD algorithm with a convex and smooth function, using ALB methodology. We recover the classical convergence result of $\calO(\slfrac{1}{K})$ \cite{beck2009gradient}.

 In light of DAGP analysis,
\eqref{eq:6} holds for GD by considering a single node, which simplifies this relation considerably.  By plugging the gradient descent dynamics in, i.e. $\bx_{k+1} = \bx_k - \mu\nabla f(\bx^v_k)$,  and defining $\tilde\bx = \bx-\bx^*$, we have 
\begin{equation}\label{eq: gd1}
    f(\bx_{k+1}) - f(\bx^*) + \frac{1}{\mu}\langle \tilde\bx_{k+1} - \tilde\bx_k, \tilde\bx_{k+1} \rangle  - \frac{L}{2}\left\|\tilde\bx_{k+1}-\tilde\bx_k\right\|_2^2 \leq 0.
\end{equation}
Here, there is no need to consider other inequalities since the setup is neither distributed nor constrained. By summing over all iterations and defining $\barby_K = \frac{1}{K} \sum_{k=0}^{K-1}\bx_k$, we have 
\begin{equation}\label{eq: gd2}
    f(\barby_K) - f(\bx^*) \leq \frac{-\bar{A}_K}{K},
\end{equation}
where the summation of the two last terms on the left-hand side of \eqref{eq: gd1} over all iterations is called $\bar{A}_K$, i.e.
\begin{equation}
    \bar{A}_K = \suml_{k=0}^{K-1} \frac{1}{\mu}\langle \tilde\bx_{k+1} - \tilde\bx_k, \tilde\bx_{k+1} \rangle - \frac{L}{2}\left\|\tilde\bx_{k+1}-\tilde\bx_k\right\|_2^2.
\end{equation}
Following the terminology of bounding $\bar{A}_K$, we can define $\bPsi_k = \left[ \tilde\bx_{k+1}  \; \tilde\bx_k \right]^T$ and the following $\bR, \bP$ and $\bS$ matrices such that \eqref{eq: linear_dynamics} and \eqref{eq: AK_def} hold.
\begin{equation}\label{eq: gd_matrices}  
    \bR = \left[ 
    \begin{array}{cc}
        0 & 0 \\
        1 & 0
    \end{array}
    \right]
    \bP = \left[
    \begin{array}{c}
         1  \\
         0
    \end{array}
    \right]
    \bS = \left[ 
    \begin{array}{cc}
        \frac{2-L\mu}{2\mu} & \frac{L\mu -1}{2\mu} \\
        \frac{L\mu-1}{2\mu} & -\frac{L}{2}
    \end{array}
    \right]
\end{equation}
The following Lemma shows if $\mu<\slfrac{1}{L}$, 
 the matrices in \eqref{eq: gd_matrices} satisfy Assumption \ref{ass: big_assumption}. Therefore, $\bar{A}_K$ has a lower bound by Lemma \ref{lem1}, and the algorithm has a convergence rate of the order $\calO(\slfrac{1}{K})$.

\begin{lemma}  
 Let $\mu<\slfrac{1}{L}$. Then, the matrices defined in \eqref{eq: gd_matrices} satisfy Assumption \ref{ass: big_assumption}.
\end{lemma}
\begin{proof}
To compute the set $\{z|\det(\bF_\beta(z)) = 0\}$, we can find $z_i$s such that
\begin{equation}
\bF_\beta(z) \left[
\begin{array}{c}
     \bE \\
     \bM
\end{array}
\right] = \left[
\begin{array}{c}
     \bO  \\
     \bO 
\end{array}
\right],
\end{equation}
for non-zero matrices $\bE$ and $\bM$. By replacing the definition of $\bF_{\beta}(z)$ from \eqref{eq: F_define}, and simple computations, we have
\begin{equation}
    \frac{1}{\beta} (z^{-1}\bI-\bR^T)^{-1} \bS (z\bI-\bR)^{-1}\bP\bP^T \; \bM = \bM. 
\end{equation}
For the sake of simplicity, by defining $\bM \coloneqq \bM_1\bM_2$, where $\bM_2 = \bP^T\bM$, we may write
\begin{equation}
    \bM_2 =  \frac{1}{\beta} \bP^T(z^{-1}\bI-\bR^T)^{-1} \bS (z\bI-\bR)^{-1}\bP \; \bM_2.
\end{equation}
Since $\bM$ is non-zero, and consequently $\bM_2$ is non-zero,  the problem changes to finding $z_i$s such that $\beta$ is an eigenvalue of $\bP^T(z^{-1}\bI-\bR^T)^{-1} \bS (z\bI-\bR)^{-1}\bP$. For the matrices defined in \eqref{eq: gd_matrices}, we have
\begin{equation}
    z^2 - 2\frac{L\mu-1-\beta\mu}{L\mu-1}z +1 = 0.
\end{equation}
This equation has two different real positive roots for every $\beta>0$ and $\mu<\slfrac{1}{L}$.
\end{proof}

\section{ Experimental Results}\label{sec: exp_results}
To evaluate the performance of DAGP in practice and compare it with the state-of-the-art algorithms, three experiments are considered. Our first experiment compares DAGP's convergence properties with those of DDPS~\cite{xi2016distributed}. 
In the second experiment, we consider the classical regularized logistic regression problem. This problem is unconstrained, and many distributed algorithms can be used to solve it with a provable geometric convergence rate. 
The last experiment demonstrates the practicality of DAGP in solving optimal transport for domain adaptation.

 In all simulations, algorithms' parameters are hand-tuned for peak performance.
In our simulations, we respectively use 
$\slfrac{\bL^{\text{in}} }{ 2d_{\text{max}}^{\text{in}} }$ and $\slfrac{\bL^{\text{out}} }{ 2d_{\text{max}}^{\text{out}} }$
as zero row sum and zero column sum matrices, where $d_{\text{max}}^{\text{in}}$ and $d_{\text{max}}^{\text{out}}$ are the largest diagonal elements of $\bL^{\text{{in}}}$ and $\bL^{\text{{out}}}$. By subtracting these matrices from the identity matrix, our row stochastic and column stochastic matrices  are computed.
We repeated each experiment multiple times, but only one instance from each experiment is presented as the difference between individual runs was minimal.

\subsection{Numerical Results, Constrained Optimization}
 
In this experiment, we assume a set of $M$ nodes collaborating to solve problem \eqref{eq: dec_constrained}, where
\begin{align}
f^v(\bx) & =  \log\big(\cosh( \ba_v^T\bx - b_v)\big), \\
S^v  & = \{ \bx \; |\;  \bc_v^T\bx - d_v \leq 0 \}.
\label{eq: synthetic}
\end{align}
The reason for the choice of the above synthetic objectives and constraints is to have functions, which are smooth, but not strongly-convex, and to have constraints, which projection onto them is simple to compute. The optimal value of the unconstrained problem is 0 if there exists a point $\tilde\bx$ such that $\bA\tilde\bx = \bb$, where $\bA=[\ba_v^T]\in\bbR^{M\times m}$, and $\bb=[b_v] \in \bbR^M$. If such a point exists, the objective functions are strongly-convex near $\tilde\bx$. Regarding the constraints, the feasible set is an intersection of $M$ halfspaces in $\bbR^m$, defined by $\bC\bx-\bd \leq \bzero$, where $\bC=[\bc_v^T]\in\bbR^{M\times m}$ and $\bd = [d_v] \in \bbR^M$.
The feasible set will become larger by increasing $m$ since the nullity of $\bC$ will increase. 
In this experiment, we study two setups. In the first setup, we set $M=10$ and $m=20$ to have a large feasible set, which contains $\tilde\bx$, and the second setup, where we set $M=20$ and $m=10$ to have a feasible set not including $\tilde\bx$.
In both setups, the nodes' local solutions and the coefficients $\ba_v, \bc_v$, and $b_v$  are generated from a zero mean and unit variance normal distribution, and $d_v$ coefficients are selected such that the feasible set not being empty.

To study the convergence properties of DAGP in comparison to DDPS, objective value and feasibility gap computed at the average of all nodes solutions $\bar\bx=\frac1M\sum_{v\in\calV}\bx^v$ are shown in Fig.~\ref{fig:toy1} and Fig.~\ref{fig:toy2}, respectively for the first and second setups. As a remark, DDPS has been proposed for an optimization problem constrained to a commonly known constraint set. In this paper, for the sake of comparison with DAGP, we modify DDPS to be applicable to the setup with distributed constraints, such that at each iteration, each node projects on its own local constraint set instead of the global constraint set.

The results indicate a linear convergence rate for the optimality gap (since the minimum is zero) in the first setup, 
which is due to the inclusion of $\tilde\bx$ in the feasible set.  On the other hand, in the second setup, some constraints are active, and since $\tilde\bx$ is not in the feasible set, the objective value is not zero. However, we note that DAGP converges to the optimal solution remarkably faster due to using a fixed step-size. Fig.~\ref{fig: fesgpN10} and Fig.~\ref{fig: fesgpN20} show that the DAGP solution rapidly moves to the feasible set, and the solution is completely in the feasible set, unlike DDPS in which the solution becomes only close to the feasible set, especially in the second setup. This shows that the modification to DDPS is not working, and our algorithm is the first working algorithm for a setup with distributed constraints and directed graphs. 

To show consensus among nodes, the squared Euclidean 
distance between $\bx^0$ and the solutions of five random nodes is depicted in Fig.~\ref{fig: cons}. 
Moreover, to show the consensus solution is also the optimal solution, $\nt{\sum_{v\in\calV}\bg^v}^2$ is computed and shown in Fig.~\ref{fig: gsum}. These metrics converge to zero and corroborate convergence to a consensus optimal solution.

\begin{figure}[!t]
	\centering
	\subfloat[Objective Value]{\includegraphics[width = 0.48\linewidth]{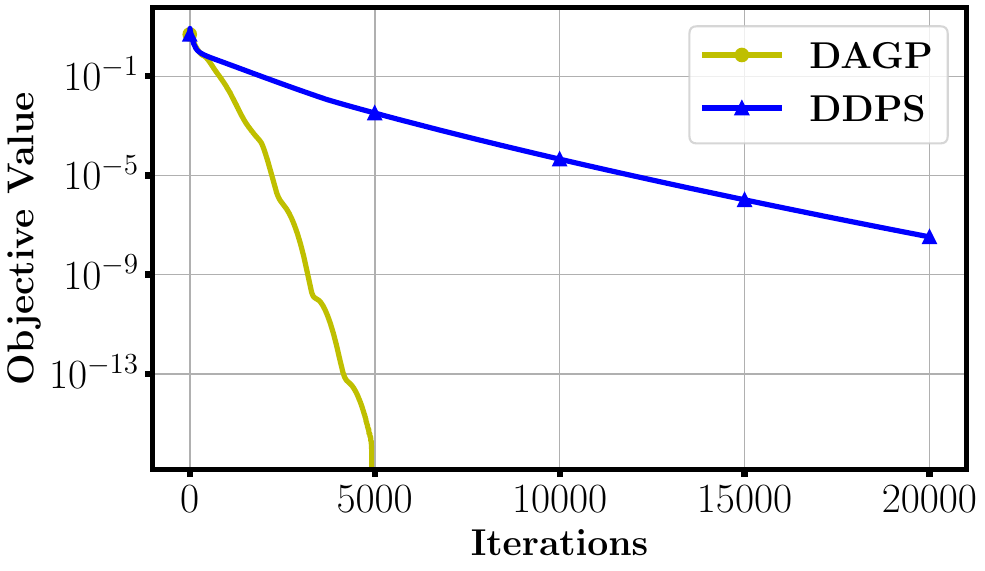} \label{fig: objectiveN10}} 
	\subfloat[Feasibility Gap]{\includegraphics[width = 0.48\linewidth]{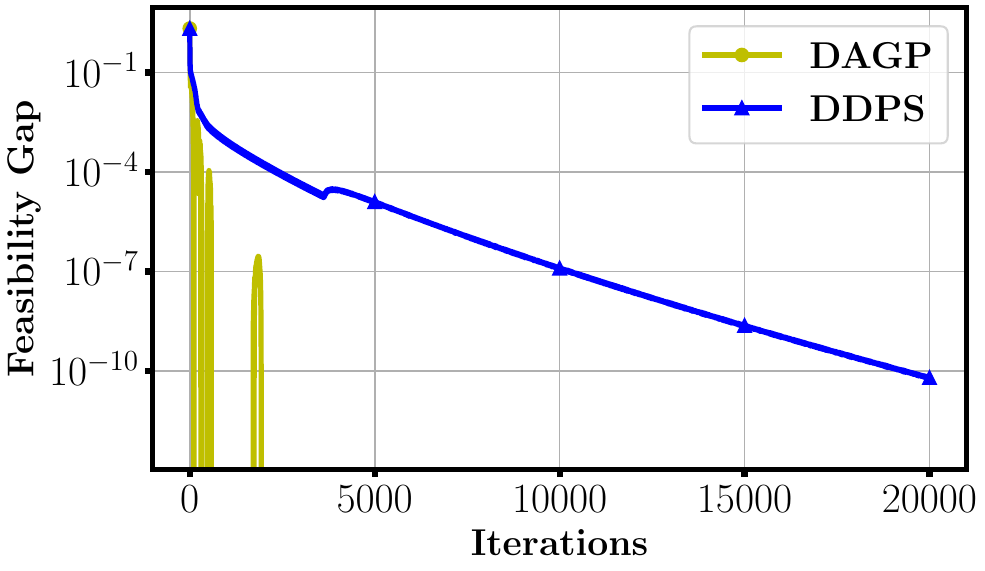}   \label{fig: fesgpN10} }

	\subfloat[Consensus solution]{\includegraphics[width = 0.48\linewidth]{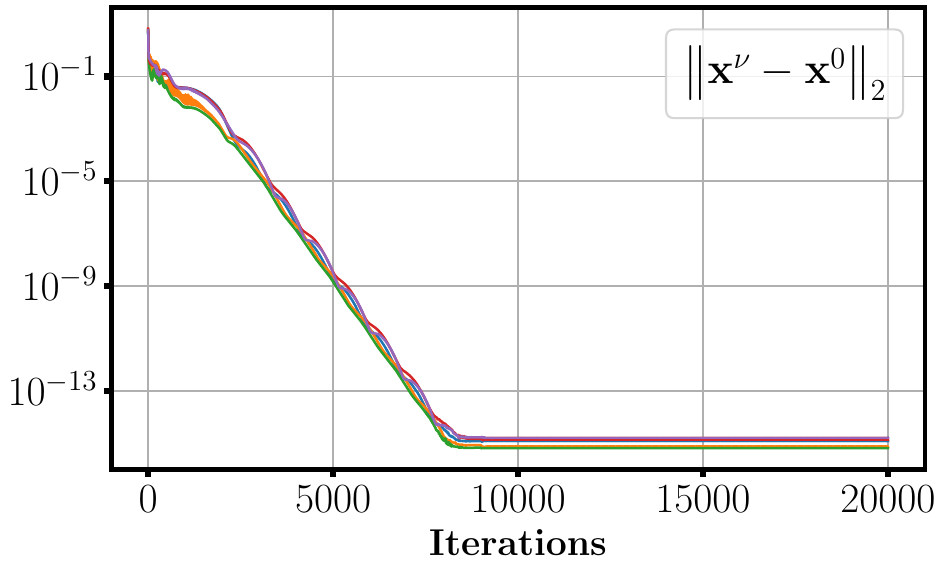}  \label{fig: cons}} 
	\subfloat[Optimal solution]{\includegraphics[width = 0.48\linewidth]{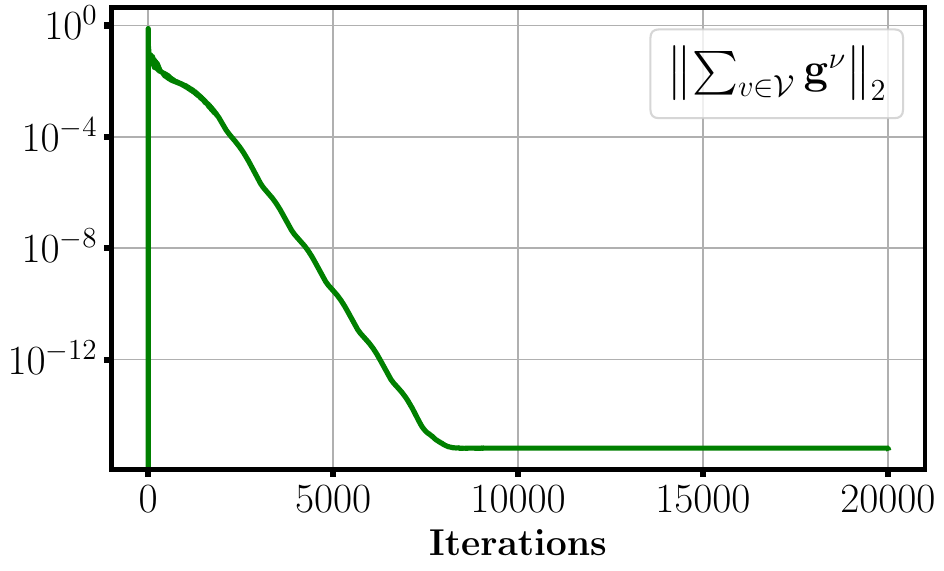} \label{fig: gsum}}
	\caption{First setup results with $m=20$ and $M=10$. Local variables move to a consensus and optimal stopping point in DAGP, while they move to  a sub-optimal point in DDPS.}
	\label{fig:toy1}
\end{figure}

\begin{figure}
    \centering
	\subfloat[Objective Value]{\includegraphics[width = 0.475\linewidth]{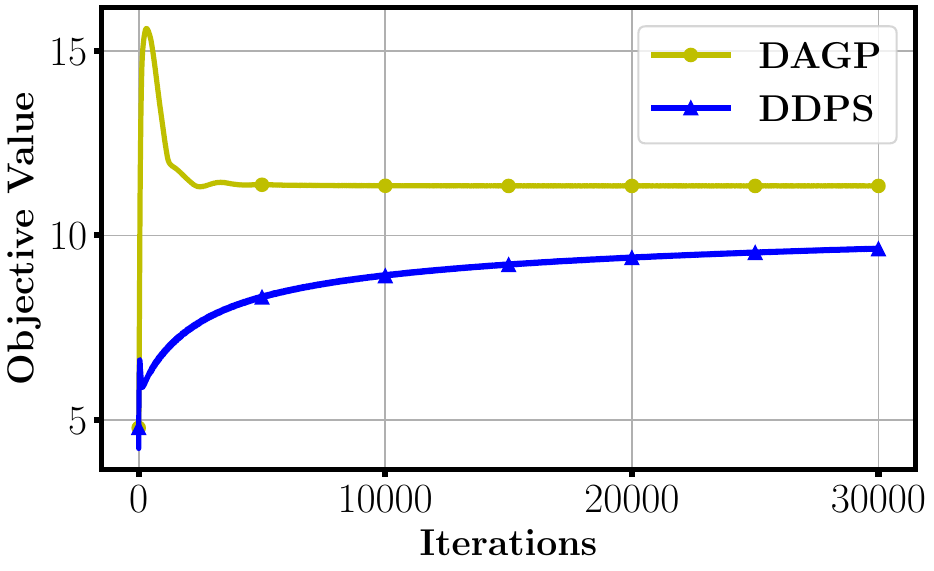} \label{fig: ObjectiveN20}} 
	\subfloat[Feasibility Gap]{\includegraphics[width = 0.49\linewidth]{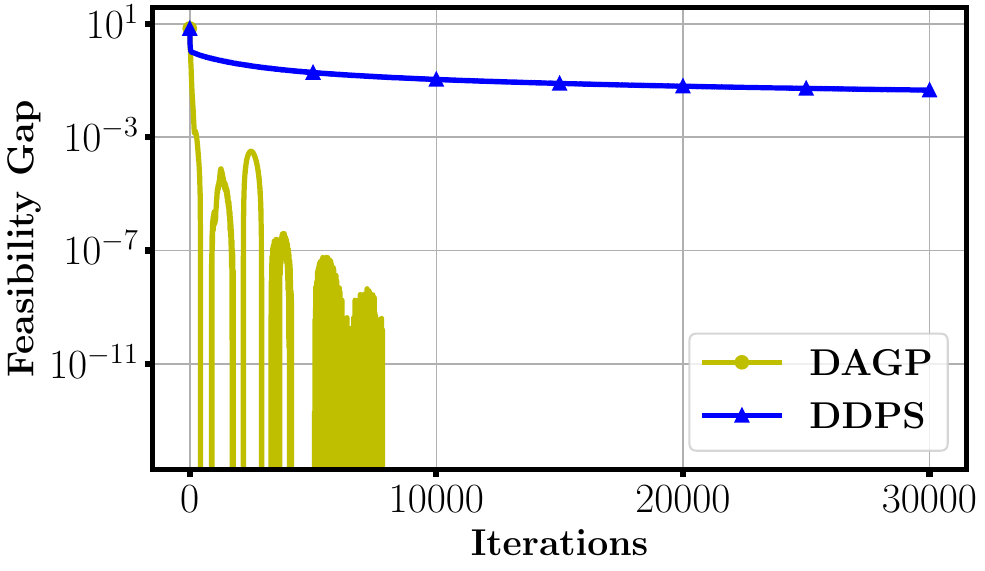} \label{fig: fesgpN20} }
    \caption{Second setup results with $m=10$ and $M=20$. As DDPS has not converged to a feasible point, it achieves a smaller function value.}
    \label{fig:toy2}
\end{figure}

\subsection{Numerical Results, Unconstrained Optimization}
We consider the classical unconstrained logistic regression problem with $\ell_2$ regularization written as
\begin{equation}
\label{eq: lr}
    \minl_{\bw} \;\; \frac{1}{N_s} \suml_{i = 1}^{N_s}  \log \left( 1 + \exp{(-y_i\bx_i^T\bw)} \right) + \frac{\lambda}{2}\nt{\bw}^2,
\end{equation}
where 
$\left\{ \bx_i, y_i \right\}_{i=1}^{N_s} \subseteq \bbR^{m} \times \{+1, -1\}$ is the set of training samples with their labels, 
$N_s$ is the total number of training samples,  
and $\lambda$ is the regularization parameter. 
In the decentralized formulation with $M$ nodes, each node local objective function can be written as
\begin{equation}
    f^v(\bw) = \frac{1}{N_v} \suml_{l=1}^{N_v} \log \left(  1+ e^{(-y_l\bx_l^T\bw)}   \right) + \frac{\lambda}{2M}\nt{\bw}^2,
\end{equation}
with $N_v$ local training samples, considering the dataset is distributed equally between nodes, i.e. $N_s = MN_v$.
In this experiment, we use $10000$ training samples corresponding to the two first digits in the MNIST dataset~\cite{lecun-mnisthandwrittendigit-2010}. We consider static, directed, and strongly-connected random graph with $M=20$ nodes as the communication network, and we set the regularization parameter to ${1}/{N_s}$. 

\begin{figure}
    \centering
    \includegraphics[width = 0.7\linewidth]{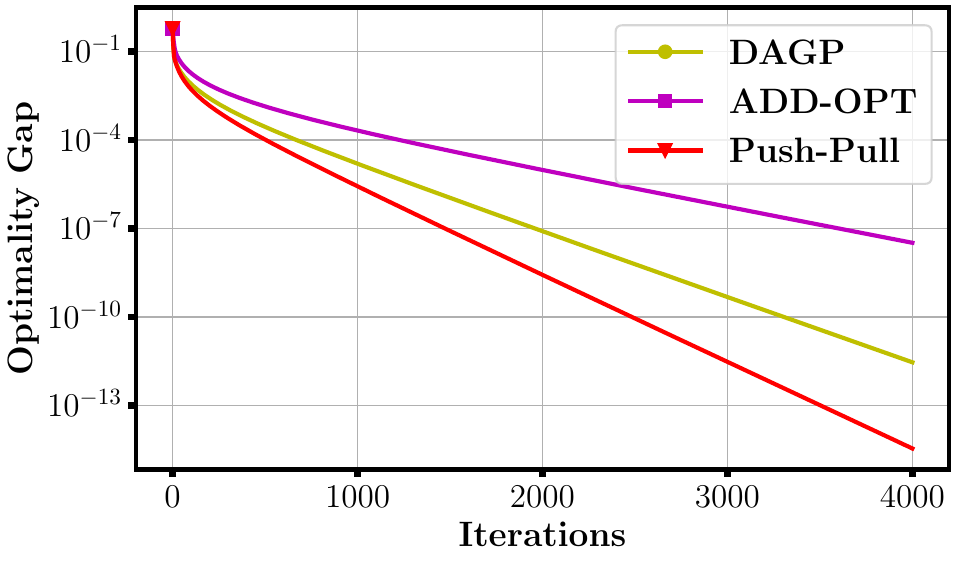}
    \caption{Convergence rate comparison of decentralized unconstrained algorithms over directed graphs. step-sizes are hand-tuned. 
    }
    \label{fig: LR}
\end{figure}

We compare the performance of DAGP with the Push-Pull and ADD-OPT algorithms. Centralized gradient descent is used to determine the optimal value $f^*$. The  optimality gap defined as $\sum_{v\in\calV} f^v(\bar\bx) - f^*$, where $\barbx$ is the average of all nodes solutions, is shown in Fig~\ref{fig: LR}. 
All the algorithms have a linear convergence rate. 
With hand-tuned step-sizes, Push-pull achieves the fastest convergence, followed by DAGP with a slight difference. ADD-OPT is the slowest one.  
Further, we observe that DAGP and Push-Pull are robust to graph connectivity, whereas ADD-OPT fails in simulations for graphs with low edge connectivity.  
The practical applicability of our algorithm is evident since it performs on par with unconstrained algorithms, and it is competitive with respect to constrained problems.

\subsection{  Application in Optimal Transport }

OT is a fundamental problem in applied mathematics with applications in control theory~\cite{chen2016optimal, chen2021optimal}, economics~\cite{galichon2018optimal}, and data science~\cite{peyre2019computational}.
The goal is to find an efficient transport plan that transports a source probability distribution $\bpi_s \in \bbR^{n_s}$ to a target probability distribution $\bpi_t\in \bbR^{n_t}$, such that the cost of transportation measured by a function $c$ is minimized. Discrete Kantorovich's formulation of OT considers the following convex optimization problem~\cite{villani2021topics}
\begin{equation}\label{eq: kantorovich}
     \argmin_{\bX}  \langle\bC,\bX \rangle \quad \sbjt \;
         \bone^T\bX = \bpi_t^T, \; \bX\bone=\bpi_s, \; x_{kl} \geq 0,
\end{equation}
where $\bX$ is a joint probability distribution, called transport plan, and $\bC$ is the cost matrix. This problem is a linear program, and its exact solution is likely sparse. The computation of the exact solution by linear programming (LP) solvers is expensive. In response, different regularized versions of OT have been proposed in the literature, e.g. the Sinkhorn algorithm \cite{cuturi2013sinkhorn}, resulting in an inexact (non-sparse) solution.
Our decentralized optimization framework is an alternative to find the exact solution of OT in a computationally efficient way.
In this regard, we reformulate \eqref{eq: kantorovich} within our framework as:
\begin{eqnarray}
\argmin_{\bX} & \; \frac{1}{2}\suml_{l=1}^{n_s} \langle \bC_{:,l},\bX_{:,l}\rangle + \frac{1}{2}\suml_{k=1}^{n_t} \langle\bC_{k,:},\bX_{k,:}\rangle, \label{eq: ot_dec}\\ \label{eq: dec_ot}
\sbjt & \;\;\;
\begin{cases} 
\bone^T_{n_s}\bX_{:,l}   =  \pi_t(l), \qquad\;\; l = 1,\dots,n_t   \\
\bX_{k,:}\bone_{n_t}     =  \pi_s(k), \qquad k = 1,\dots,n_s  \\
x_{kl} \geq 0, \qquad\qquad\quad\;\;\:\: \forall l,k
\end{cases}\nonumber 
\end{eqnarray}
where $\bX_{:,l}$ shows the $l^\tth$ column, while $\bX_{k,:}$ shows the $k^\tth$ row of the matrix. 
The objective function is linear, hence, smooth and convex. 
Moreover, the constraints are convex probability simplicies, whose orthogonal projection has been extensively studied~\cite{condat2016fast}. Therefore this problem satisfies our setup assumptions.

Our experiments on OT is divided into two parts. First, we show the computational efficiency of decentralized algorithms, specifically DAGP,  to solve the optimal transport problem~\eqref{eq: ot_dec}. Then, we demonstrate its practicality for domain adaptation on real datasets. 
In the first part, we consider $n_s=n_t=n$ bins that approximate normal source and target distributions, which are respectively $\calN({1}/{3}, {1}/{4})$ and $\calN(\slfrac{2}{3}, \slfrac{1}{8})$, in the interval of $[0,1]$. 
The Euclidean distance is used as a cost metric.  We find the optimal plan using DAGP algorithm over a fully-connected network with $M$ agents, as well as in a single-machine setup using the default solver from the Python CVXOPT library~\cite{cvxopt}.
The objective and constraints in the decentralized setup are distributed such that each node has access to several rows or columns in a way that the number of terms at each node is minimized. 
Decentralized simulations are performed in a synthetic environment with ideal communications,  and the algorithms are assumed to converge if the change in the objective is less than $10^{-7}$ and the distance to every  constraint is less than $10^{-4}$. In the centralized setting, we relax the equality constraints to help the linear program solver find a feasible solution, i.e. the equality constraints become
\begin{alignat*}{2}
    & \left| \suml_{k=1}^{n_s} x_{kl} - \pi_t(l) \right| < 10^{-4},  && \qquad  l=1,\dots,n_t \\
    & \left| \suml_{l=1}^{n_t} x_{kl} - \pi_s(k) \right| < 10^{-4}.  && \qquad  k=1,\dots,n_s 
\end{alignat*}
As a rough measure of the complexities of the algorithms, their run time averaged  over $100$ experiments is reported in Fig.~\ref{fig: run_time}. 
We observe that solving the optimal transport problem using  linear programming solvers becomes impractical for large $n$ even by relaxing the equality constraints, while the decentralized DAGP method still performs in a reasonable amount of time. 

\begin{figure}
    \centering
    \includegraphics[width = 0.7\linewidth]{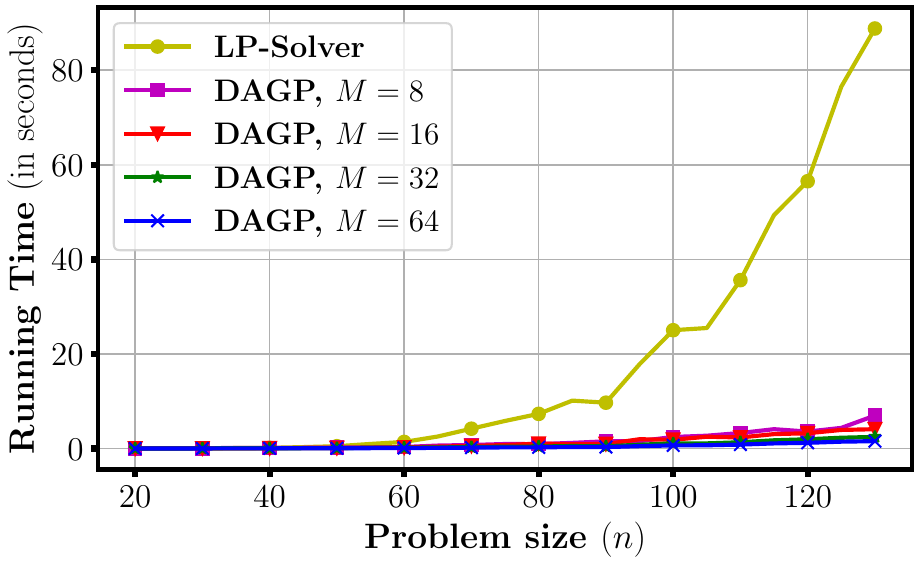}
    \caption{Average required time to solve the optimal transport problem in single-machine and decentralized settings. Communications are ideal. }
    \label{fig: run_time}
\end{figure}

In the second part, we consider the application of optimal transport in domain adaptation \cite{7586038}. 
%
In this experiment, we consider two classical datasets used for domain adaptation, namely MNIST~\cite{lecun-mnisthandwrittendigit-2010} and USPS~\cite{uspsdataset}. We reshape all images into vectors of size $256$, and sample $10000$ data points from the MNIST dataset for better visualizations. We use the Euclidean distance as the cost metric.  Considering the distribution of data points is uniform, we compute the optimal transport plan using DAGP and Sinkhorn~\cite{cuturi2013sinkhorn} algorithms. 
The resulting OT plans are depicted in Fig.~\ref{fig: OT} with a precision of $10^{-5}$.
DAGP solves \eqref{eq: ot_dec} to find a sparse transport plan, while the Sinkhorn algorithm solves the neg-entropy regularized version of the problem and clearly leads to a lower level of sparsity in the transport plan, making the classes less distinguishable.

To conclude, in applying optimal transport for domain adaptation, DAGP finds an exact sparse transport plan, and it can be easily scaled by introducing more agents. At the same time, CVXOPT solvers are impractical due to the size of the problem, as shown in the first part of the experiment, and computationally efficient algorithms like Sinkhorn find inexact solutions to compensate for slow convergence.

\begin{figure}[t!]    
    \centering
    \includegraphics[width = 0.9\linewidth]{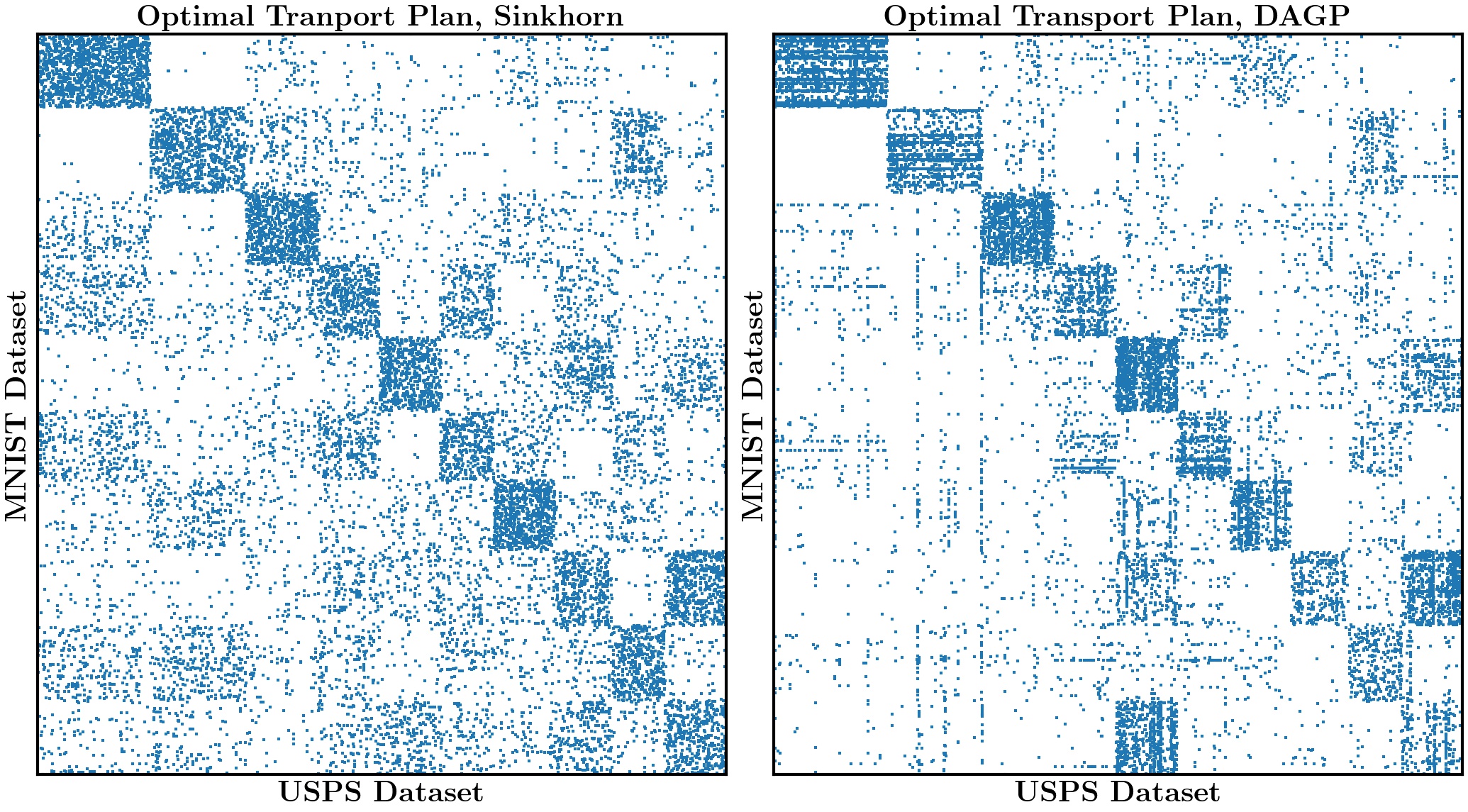}
    \caption{Optimal transport plans computed by DAGP and Sinkhorn algorithms. 
    DAGP finds a sparser transport plan since it solves the OT problem without a neg-entropy regularizer. 
    }
    \label{fig: OT}
\end{figure}

\section{Conclusion}
In this paper, we proposed a decentralized algorithmic solution to a general convex constrained optimization problem with distributed constraints by presenting the DAGP algorithm. We introduced the aggregate lower-bounding convergence analysis framework and proved that DAGP converges to a consensus optimal solution for convex and smooth objective functions over directed networks. The convergence rates were $\calO({1}/{K})$ for the feasibility gap and $\calO({1}/{\sqrt{K}})$ for the optimality gap. We showed that our proof technique could be used for analyzing a wide range of algorithms without searching for a suitable Lyapunov function.
Moreover, we presented a novel proximal backtracking algorithm for an efficient projection onto the epigraph of convex functions, which can be combined with DAGP to solve non-differentiable problems. Experimentation for this scenario is postponed to a future study. 
We presented an application of DAGP to exactly solving large-scale optimal transport problems in a decentralized manner.
Numerical results also demonstrated that DAGP outperforms DDPS, and competes with the state-of-the-art Push-Pull algorithm in the unconstrained logistic regression problem.

\appendices

\section{Proof of Lemma \ref{lem1}}\label{appendix: proofLem1}
 We start by stating several useful lemmas. Then, we present the proof of Lemma \ref{lem1}.

\begin{lemma} \label{lem: solution_form}
For given matrices  $\{\bA_n\}_{n=0}^N$ and $\{\bB_n\}_{n=0}^N$, define
\begin{equation}
\label{eq: F_lemma1}
    \bar\bF(z) = 
    \left[
     \begin{array}{cc}
     \suml_{n=0}^{N}\bA_n\;z^{n} \;& \suml_{n=0}^{N} \bB_n\;z^{n} \\
     \suml_{n=0}^{N} \bC_n\; z^{n} \;& \suml_{n=0}^{N} \bD_n\; z^{n}
      \end{array}
    \right]
\end{equation}
Take $Z \coloneqq \left\{ z\neq 0 | \det(\bar\bF(z)) = 0  \right\}$ and assume that any $ z_i \in Z$ is a simple root of $\det(\bar\bF(z))$.
Furthermore, assume that $z\bar\bF^{-1}(z)$ has a limit on the complex plane as $z\to 0$. Denote by
$\bn_i \coloneqq [ \bn_{i,\bpsi}^T \;\; \bn_{i,\blambda}^T ]^T $
the unique normalized vector spanning the null space of $\bar\bF(z_i)$%
. Then,
 any solution of the following system of linear recurrence equations for $k=0,1,2,\ldots$
\begin{equation}\label{eq: sys_recurrence}
    \begin{aligned}
    \sum_{n=0}^N \bA_n\bPsi_{k+n} + \bB_n\bLambda_{k+n}  & = \bO \\
    \sum_{n=0}^N \bC_n\bPsi_{k+n} + \bD_n\bLambda_{k+n}  & = \bO, 
    \end{aligned}
\end{equation}
 can be written as
\begin{equation}\label{eq: psi_lambda_sol}
\begin{aligned}
    \bPsi_k & = \sum_{i=1}^p \bgamma_i^T \otimes \bn_{i,\bpsi} \; z_i^k\\
    \bLambda_k & = \sum_{i=1}^p \bgamma_i^T \otimes \bn_{i,\blambda} \; z_i^k,
\end{aligned}
\end{equation}
where $p$ is the cardinality of the set $Z$,  $\bgamma_i$ is an arbitrary vector, and $\otimes$ shows the Kronecker product.  
\begin{proof}
Consider infinite-duration one-sided signals $\bPsi_k$ and $\bLambda_k$. Taking the one-sided $Z$-transform of recurrence equations in \eqref{eq: sys_recurrence}, we have
\[
\bar\bF(z) \left[ 
\begin{array}{c}
     \bPsi(z) \\
     \bLambda(z) 
\end{array}
\right]
 =z\bU(z), 
 \]
 \[
 \bU(z):=\suml_{n=0}^N\suml_{k=0}^{n-1}z^{n-k-1} \left[
     \begin{array}{cc}
     \bA_n &  \bB_n \\
      \bC_n &  \bD_n
      \end{array}
    \right]\left[ 
\begin{array}{c}
     \bPsi_k \\
     \bLambda_k 
\end{array}
\right]
 \]
where $\bar\bF(z)$ is defined in \eqref{eq: F_lemma1} and the entries of $\bU(z)$ are polynomials. By the assumptions for the null space of $\bar\bF(z)$, we conclude that there are only simple nonzero poles at $z=z_i\in Z$ in the entries of $\bar\bF(z)$ and hence by inverse Z-transform, we obtain that
\begin{equation}
\left[\begin{array}{c}
    \bPsi_k  \\
    \bLambda_k  
\end{array}\right]= \sum_{i=1}^p 
    \bu_i z_i^k
\end{equation}
where $\bu_i$ are suitable constant matrices. Replacing this expression in the recurrent equation, we obtain that
\begin{equation}
    \sum_{i=1}^p 
    \bar\bF(z_i)\bu_i z_i^k=\bzero
\end{equation}
As $z_i$ are distinct and $k$ is arbitrary, we conclude that $\bar\bF(z_i)\bu_i=\bzero$, which yields the result. 
\end{proof}
\end{lemma}

\begin{lemma}\label{lem: det}
Let $z_i$ for $i=1,2\ldots,p$ be real. 
Consider a matrix $\bB$ in the following block form
\begin{equation}
    \bB=\left[\begin{array}{c}
        \bB_{1}  \\
        \hline
         \bB_{2}
    \end{array}\right],
\end{equation}
where $\bB_1,\bB_2$ are two real $m\times p$ matrices. Correspondingly, define 
matrices $\bar\bT_K$ with the following block structure
\begin{equation}
\bar\bT_K = \bB \circ
    \left[
    \begin{array}{ccc}
         z_1^{K}\bone_m & \dots & z_{p}^K\bone_m \\
         \hline
         &    \bone_{m\times p}       &
    \end{array}
    \right]
\end{equation}
where $\circ$ denotes the Hadamard  product. There exists a continuous function $ K_0 =  K_0 (\bB, z_1, z_2,\ldots, z_p) $ such that if $\bar\bT_0$ and $\bar\bT_1$ have full column rank, then  $\bT_K$ also has full column rank for $K>K_0$.
\end{lemma}
\begin{proof}
We divide the case of odd and even $K$. For even $K$, i.e. $K=2K'$, note that since $\bar\bT_0$ has full column rank, $2m\geq p$ and we may append suitable $2m-p$ columns to $\bar\bT_{K}$ to obtain a square $2m\times 2m$ matrix $\tilde\bT_K$ such that $\tilde\bT_0$ is full rank.
We show that the determinant of  $\tilde\bT_K$ for sufficient large even $K$ cannot be zero, which also implies that $\bar\bT_K$ is full-column-rank. 
Remember the general formula of determinant as
\begin{equation}\label{eq:det}
    \det(\tilde\bT_K) = \suml_{(2m)! \; \text{terms}} \pm\;  \tilde t_{1j_1}\tilde t_{2j_2}\dots \tilde t_{(2m)j_{2m}}, 
\end{equation}
where $(j_1,j_2,\ldots,j_{2m})$ is any permutation of $(1,2,\ldots,2m)$. Hence, all the terms in the summation can be written as a coefficient multiplied by $(z_{j_1}z_{j_2}\dots z_{j_{m}})^{2K'}$:
\begin{equation}
    \det(\tilde\bT_K)=\suml_{l}b_lZ_l^{K'}
\end{equation}
where each $Z_l$ denotes a unique positive value of the combination $(z_{j_1}z_{j_2}\dots z_{j_{m}})^{2}$.
Note that by the assumption, it holds 
that 
$\det(\tilde\bT_0)=\sum b_l\neq 0$. Hence, at least one of the $b_l$s is nonzero. With an abuse of notation,  take $Z_1$ the largest $Z_l$ with a nonzero $b_l$.
If this is the only nonzero term, the result holds for every $K'\geq 0$. If not, denote by $Z_2$ the second largest combination with a nonzero coefficient. Note that  $Z_2<Z_1$. Moreover, by the triangle inequality, the sum of all terms in \eqref{eq:det} except $b_1Z_1^K$  is bounded by $\left|
    b_2Z_2^K
    \right| 
    \left[\binom{2m}{m}-1\right]$ and hence
\begin{equation}
    |\det(\bar\bT)|\geq \left| 
    b_1Z_1^K 
    \right| 
    -
    \left|
    b_2Z_2^K
    \right| 
    \left[\binom{2m}{m}-1\right].
\end{equation}
We note that the right-hand side is strictly positive for a sufficiently large $K'$
\begin{equation}\label{eq: k_bound}
K'> 
\frac{ \log{\left|\frac{b_2}{b_1}\right|} + m\log{2} }  {\log{\left|\frac{Z_1}{Z_2}\right|}}.
\end{equation}
This completes the proof for even $K$. For odd $K=2K'+1$, append $\bar\bT_K$ with proper columns such that $\bar\bT_1$ is full rank. Again, the determinant can be written as 
\begin{equation}
    \det(\tilde\bT_K)=\suml_{l}b'_lZ_l^{K'}
\end{equation}
where $Z_l$ are as before, but the coefficients $b'_l$ are different. However, from the assumption $\det(\bar\bT_1)=\sum_lb'_l\neq 0$ and the rest of the argument for even $K$ holds again. The continuity of $K'$ follows the construction.
\end{proof}

\begin{lemma}\label{lem: zero_sol_sys}
Let Assumption \ref{ass: big_assumption} holds for some $\bR,\bS$ and $\bP$ matrices.
Then, for $C>-\lambda_{\mathrm{min}}(\bS)$  and a sufficiently large $K$, the following system of linear recurrent equations
\begin{equation}\label{eq: sys_Time_period}
    \begin{aligned}
        \bPsi_{k+1} - \bR\bPsi_k - \frac{1}{\beta}\bP\bP^T\bLambda_k  = \bO & \qquad  k = 0,\ldots,K-2\\
        \bS\bPsi_{k+1} - \bR^T\bLambda_{k+1} + \bLambda_k  = \bO & \qquad  k=0,\ldots,K-2        
    \end{aligned}
\end{equation}
with the following boundary conditions
\begin{gather}\label{eq: boundary_cond}
\begin{aligned}
    &\bLambda_{K-1}  = \bO \\
    &\left(\bS + (C+\beta)\bI\right)\bPsi_0 - \bR^T\bLambda_0  = \bO
\end{aligned}
\end{gather}
has no non-zero solution for $\bPsi_k$ and $\bLambda_k$ and every $\beta$. 
\end{lemma} 
\begin{proof}

The general solution of the system defined in \eqref{eq: sys_Time_period} is in the form of \eqref{eq: psi_lambda_sol} by Lemma \ref{lem: solution_form}. 
This solution should satisfy the boundary conditions. By replacing
\begin{gather}\label{eq: boundary}
    \bLambda_0 = \sum_{i=1}^p  \bgamma_i^T\otimes \bn_{i,\blambda}, \qquad
    \bPsi_0    = \sum_{i=1}^p  \bgamma_i^T\otimes \bn_{i,\bpsi}, \\
    \bLambda_{K-1} = \sum_{i=1}^p \bgamma_i^T\otimes \bn_{i,\blambda}\;z_i^{K-1},
\end{gather}
in \eqref{eq: boundary_cond}, we have the following system of equations to find $\bgamma_i$s.
\begin{equation}\label{eq: sys_gamma}
\underbrace{
    \left[
    \begin{array}{ccc}
         \dots & \bn_{i,\blambda}z_i^{K-1} & \dots  \\
         \dots & (\bS+(C+\beta)\bI)\bn_{i,\bpsi}-\bR^T\bn_{i,\blambda} & \dots 
    \end{array}
    \right]}_{\bT_\beta} \left[
    \begin{array}{c}
           \vdots \\
           \bgamma_{i}^T \\
           \vdots
    \end{array}
    \right] = \bO
\end{equation}
Note that by the assumptions, $\bT_0=\lim_{\beta\to 0}\bT_{\beta}$ exists. Furthermore,
for $C\geq -\lambda_{\min}(\bS)$ and  every $\beta >0$, $\bS + (C+\beta)\bI$ is full-rank. 
Then, matrix $\bT_\beta$ satisfies the conditions of Lemma \ref{lem: det} for every $\beta\in (0\ \infty]$. Hence, it has non-zero determinant for  $K>K_{\beta}$ where $K_{\beta}$ is a continuous function of $\beta$. Since $K_0,K_\infty$ are finite, we conclude that $K_\beta$ has a finite upper bound $\bar{K}$ and for $K>\bar{K}$,  $\bgamma_i = \bzero$, for all $i$. 
\end{proof}

\begin{lemma}
    Let the conditions of Lemma \ref{lem: zero_sol_sys} are satisfied for some known matrices $\bR,\bS$ and $\bP$. Then,
    there exists a $C\geq0$ such that for every $\beta>0$, the following system of linear recurrences has no non-zero solution for $\{\bPsi_k,\bLambda_k,\tlbX_{k+2}\}$, considering the boundary conditions $\bLambda_{-1} = \bLambda_{K-1} = \bO$.
    \begin{align}
    \bLambda_{k-1}-\bR^T\bLambda_k+(\bS+\left(C+\beta\right)\delta_{k,0}\bI)\bPsi_{k}=\bO  & \qquad k=0,1,\ldots,K-1 \label{eq:10}  \\
    \bPsi_{k+1} - \bR\bPsi_k - \bP\tilde\bX_{k+2} =\bO  &\qquad k=0,1,\ldots,K-2  \label{eq:44}\\
    \bP^T\bLambda_k - \beta\tlbX_{k+2} = \bO &\qquad  k=0,1,\ldots, K-2 \label{eq:13}  
    \end{align}
\end{lemma}

\begin{proof}
    The boundary condition $\bLambda_{-1}=\bO$ is equivalent to $(\bS+\left(C+\beta\right)\bI)\bPsi_{0}-\bR^T\bLambda_0=\bO$, together with starting the recursive equation \eqref{eq:10} from $k=1$. Then, \eqref{eq:10} can be rewritten as
    \begin{equation}
        \bLambda_{k}-\bR^T\bLambda_{k+1}+\bS\bPsi_{k+1}=\bO.  \qquad k=0,1,\ldots,K-2 \nonumber
    \end{equation}
    Moreover, one can eliminate $\tilde\bX_{k+2}$ from the system, by calculating it in \eqref{eq:13} and replacing it in the other relations. This will lead to the system of recurrences presented in \eqref{eq: sys_Time_period}, with the same boundary conditions as \eqref{eq: boundary_cond}. Then, the proof follows Lemma \ref{lem: zero_sol_sys}.
\end{proof}

Now, we present the proof of Lemma \ref{lem1}. Note that the claim is equivalent to the statement that zero is the optimal value for the optimization problem
    \begin{alignat}{2}
        & \minl_{\substack{\{\bPsi_k\}_{k=0}^{K-1}, \\ \{\tlbX_{k+2}\}_{k=0}^{K-2}}} && \quad \frac{1}{2}\suml_{k=0}^{K-1}\langle\bPsi_k,\bS\bPsi_k\rangle+\frac{C}{2}\|\bPsi_0\|^2_{\mathrm{F}} \nonumber\\
       &\qquad \sbjt && \quad \bPsi_{k+1}=\bR\bPsi_k+\bP\tlbX_{k+2}.  
           \qquad k=0,1,\ldots, K-2 \label{eq: lower_bounding_probelm}
    \end{alignat}
    If the claim does not hold, the optimization is unbounded and the following restricted optimization will achieve a strictly negative optimal value at a non-zero solution.
    \begin{alignat}{2}
        &\minl_{\substack{\{\bPsi_k\}_{k=0}^{K-1}, \\ \{\tlbX_{k+2}\}_{k=0}^{K-2}}} && \quad \frac{1}{2}\suml_{k=0}^{K-1}\langle\bPsi_k,\bS\bPsi_k\rangle+\frac{C}{2}\|\bPsi_0\|^2_{\mathrm{F}} \nonumber \\
        &\qquad \qquad\sbjt &&  \quad 
           \bPsi_{k+1}=\bR\bPsi_k+\bP\tlbX_{k+2}, \qquad k=0,1,\ldots, K-2 \nonumber \\
        & && \quad \frac{1}{2}\nf{\bPsi_0}^2 + \frac{1}{2}\suml_{k=0}^{K-2}\|\tlbX_{k+2}\|_\mathrm{F}^2 \leq \frac{1}{2}. \nonumber
    \end{alignat}
    Such a solution satisfies the KKT condition,  which coincides with \eqref{eq:10}, \eqref{eq:13} where $\{\bLambda_k\},\beta\geq 0$ are dual (Lagrangian) multipliers corresponding to the constraints. We also observe that the optimal value at this point is given by $-\beta\left( \nf{\bPsi_0}^2 + \sum_{k=0}^{K-2}\|\tlbX_{k+2}\|_\mathrm{F}^2\right)$. This shows that $\beta>0$. This contradicts the assumption that such a point does not exist and completes the proof.

Note that the above result resembles the standard quadratic control setup but is different from it as it is not assumed that $\bS$ is a positive-semi-definite matrix.

\section{Epigraph Projection Operators}\label{appendix: EPO}
In this Appendix, we show that the problem of evaluating the EPO of a convex function $f$ can always be reduced to calculating the underlying proximal operator of $f$, together with calculating an adaptive step-size:
\begin{theorem}\label{thm:reduction}
For every convex lower semi-continuous (LSC)\footnote{This assumption does not restrict  optimization problems. With this assumption, the EPO of $f$ denoted by $\mathrm{P}_{\epi(f)}(\bx,t)$ is the orthogonal projection operator onto $\epi(f)$, which always exists and is unique.  } function $f$, we have
\[
\mathrm{P}_{\epi(f)}(\bx,t)=(\prox_{\tau f}(\bx), t+\tau),
\]
where $\tau$ is either the unique non-negative solution of the equation $f(\prox_{\tau f}(\bx))=t+\tau$, if $t<f(\bx)$, or $\tau=0$, otherwise. Equivalently, $\tau$ is the unique solution of the following concave optimization problem:
\begin{equation}\label{eq:backtracking}
\maxl_{\tau\geq 0}\;\;\left(\tlf^\tau(\bx)-t\right)\tau-\frac 1 2\tau^2,
\end{equation}
where $\tlf^\tau$  is the  Moreau envelope of $f$ with parameter $\tau$.
\end{theorem}
\begin{proof}
By definition, we have
\begin{alignat}{3}
&\mathrm{P}_{\epi(f)}(\bx,t)= && \argmin\limits_{\bx^\prime,t^\prime} && \;\; \frac 1 2\|\bx-\bx^\prime\|^2+\frac 1 2(t-t^\prime)^2 \nonumber\\
& && \qquad\; \sbjt &&\quad t^\prime\geq f(\bx^\prime) \label{eq:primal}
\end{alignat}
Since the epigraph is closed and convex, the solution exists and is unique. Now, consider the following Lagrange dual form
\begin{eqnarray}\label{eq:dual}
\max_{\tau\geq 0}\;\;\minl_{\bx^\prime,t^\prime} \;\; \frac 1 2\|\bx-\bx^\prime\|^2+\frac 1 2(t-t^\prime)^2+\tau\left(f(\bx^\prime)-t^\prime\right).
\end{eqnarray}
The solution to the inner optimization is given by $\bx^\prime=\prox_{\tau f}(\bx),\ t^\prime=t+\tau$. By replacement, we obtain \eqref{eq:backtracking} as the outer optimization, which has the unique solution $\tau$, satisfying the conditions of Theorem 2. Finally, since $\tau, (\bx^\prime,t^\prime)$ are the unique solutions of the dual form in \eqref{eq:dual}, they are the solution of \eqref{eq:primal}, by the saddle point optimality theorem. 
\end{proof}

As a remark, \cite{chierchia2015epigraphical } and \cite[chapter~6.6.2]{parikh2014proximal} have introduced other types of reduction for EPOs that involve proximal operators. These reductions are restricted to  solving complex equations or computing the proximal operator of more complicated functions than the underlying objective functions $f^v$. 
For this reason, we observe that our reduction scheme can be applied more conveniently as it merely involves the proximal operator of the objective term.
In particular, Theorem~\ref{thm:reduction}  suggests a proximal backtracking scheme for implementing EPOs by applying the gradient ascent algorithm to solve the optimization problem in \eqref{eq:backtracking}.  We present the resulting procedure in Algorithm~\ref{alg:backtracking}. We note that since the objective function in \eqref{eq:backtracking} is 1-strongly concave, any choice of a step-size $\mu\in (0,1)$ leads to linear  convergence. Due to strong concavity, more elaborate optimization schemes can also be considered, which are left for future studies. 
\begin{algorithm}
\begin{algorithmic}
\STATE {\bf input:} $(\bx,t)$ and a
convex LSC function $f$ 
\IF{$t\geq f(\bx)$}
    \STATE Return $(\bx^\dagger,t^\dagger)=(\bx,t)$
\ELSE
\STATE {\bf set:} $\mu\in (0,1)$ 
\STATE {\bf initialize:} $\tau\geq 0$ 
\REPEAT
\STATE Set $(\bx^\dagger,t^\dagger)=(\prox_{\tau f}(\bx),t+\tau)$
\STATE Update $\tau$ to $\tau+\mu(f(\bx^\dagger)-t^\dagger)$
\UNTIL{convergence}
\ENDIF
\STATE Return $(\bx^\dagger,t^\dagger)$
\end{algorithmic}
\caption{Epigraph Projection by Proximal Backtracking}
\label{alg:backtracking}
\end{algorithm}

\bibliographystyle{IEEEtran}
\bibliography{arxiv.bib}




\end{document}